\documentclass[a4paper]{article}
\usepackage[T1]{fontenc}
\usepackage{amsmath}
\usepackage{amsthm}
\usepackage{amssymb}
\usepackage{graphicx}
\usepackage{authblk}
\usepackage{hyperref}
\usepackage{multirow}
\usepackage{subfigure}
\usepackage{cite}
\usepackage{color} 
\usepackage[dvipsnames]{xcolor}
\usepackage{comment}
\usepackage{enumitem}
\newtheorem{theorem}{Theorem}[section]
\newtheorem{proposition}[theorem]{Proposition}
\newtheorem{lemma}[theorem]{Lemma}
\newtheorem{assumption}{Assumption}

\newtheorem{problem}{Problem}
\newtheorem{definition}{Definition}
\newtheorem{remark}{Remark}

\usepackage{fullpage}
\newcommand{\eqlab}[1]{\label{eq:#1}}
\renewcommand{\eqref}[1]{(\ref{eq:#1})}

\newcommand{\Wert}{{\vert\kern-0.25ex\vert\kern-0.25ex\vert}}

\newcommand{\figref}[1]{Fig.~\ref{fig:#1}}
\newcommand{\figlab}[1]{\label{fig:#1}}
\newcommand{\propref}[1]{Proposition~\ref{proposition:#1}}
\newcommand{\proplab}[1]{\label{proposition:#1}}
\newcommand{\lemmaref}[1]{Lemma~\ref{lemma:#1}}
\newcommand{\lemmalab}[1]{\label{lemma:#1}}

\newcommand{\remlab}[1]{\label{remark:#1}}
\newcommand{\thmref}[1]{Theorem~\ref{theorem:#1}}
\newcommand{\thmlab}[1]{\label{theorem:#1}}

\newcommand{\defnlab}[1]{\label{defn:#1}}
\newcommand{\defnref}[1]{Definition~\ref{defn:#1}}
\newcommand{\appref}[1]{Appendix~\ref{app:#1}}

\newcommand{\assumpref}[1]{Assumption~\ref{assumption:#1}}
\newcommand{\assumplab}[1]{\label{assumption:#1}}
\newtheorem{cor}[theorem]{Corollary}
\newcommand{\corlab}[1]{\label{cor:#1}}
\newcommand{\corref}[1]{Corollary~\ref{cor:#1}}
\newcommand{\secref}[1]{section~\ref{sec:#1}}
\newcommand{\seclab}[1]{\label{sec:#1}}
\newcommand\rspp[1]{{\color{black}{#1}}}

\newcommand\rsp[1]{{\color{black}{#1}}}
\newcommand\rspf[1]{{\color{black}{#1}}}
\author[1]{Peter De Maesschalck}
\author[2]{Kristian Uldall Kristiansen}

\affil[1]{Hasselt University, Campus Diepenbeek, Agoralaan Gebouw D, 3590 Diepenbeek, Belgium}
\affil[2]{Technical University of Denmark, Lyngby, 2800 Kgs. Lyngby, Denmark}
\affil[1]{{\tt peter.demaesschalck@uhasselt.be}}
\affil[2]{{\tt krkri@dtu.dk}}

\title{On $k$-summable normal forms of vector fields with one zero eigenvalue}

\begin{document}
\maketitle
\date{}

\begin{abstract}\noindent
In this paper, we study normal forms of analytic saddle-nodes in $\mathbb C^{n+1}$ with any Poincar\'e rank $k\in \mathbb N$. The approach and the results generalize those of Bonckaert and De Maesschalck from 2008 that considered $k=1$. In particular, we introduce a Banach convolutional algebra that is tailored to study differential equations in the Borel plane of order $k$.  {One of the subtleties that we take care of in this paper, is that nontrivial Jordan blocks are allowed in the linear part of the vector field.} 
We anticipate that our approach can stimulate new research and be used to study different normal forms in future work.
\end{abstract}
\textit{Keywords:} normal forms, center manifolds, Gevrey properties, summability, saddle-nodes;
\newline
\textit{2020 Mathematics Subject Classification:} 34C23, 34C45, 37G05, 37G10


\section{Introduction}



The reference \cite{bonckaert2008a} studied normal forms of analytic vector fields in $\mathbb C^{n+1}$, for which the origin is an isolated singularity with one single zero eigenvalue of the linearization. In particular, the authors proved -- under the additional assumption that all other eigenvalues were distinct and satisfied a nonresonance condition -- that formal changes of coordinates to normal form satisfy Gevrey growth estimates.   Note that in the paper, the authors used a blow-up transformation (and a regular coordinate change combined with a change of time) to arrive at the pre-normal form:
\begin{equation}\eqlab{prenormalform}
\begin{aligned}
 \dot{x}=\frac1k x^{k+1},\qquad \dot{y} &= A y+x f(x,y).
\end{aligned}
\end{equation}

The constant $k$ is called the Poincar\'e rank, see \cite{balser2000a}.
When $k=1$ the authors of \cite{bonckaert2008a} also proved (under stronger resonance conditions) that the formal normal form transformation is $1$-summable. For this statement, the authors set up a     fix-point equation for the Borel-transform of the normalizing transformation, working on spaces of analytic functions $H=H(w)$ that have at most exponential growth (of order $1$), i.e.~there exists a $\mu>0$ and a $C>0$ such that $$\vert H(w) \vert e^{-\mu \vert w\vert} \le C\quad \forall\,w\in S(\theta,\alpha). $$ Here \begin{align*}
  S(\theta,\alpha):=\{x\in \mathbb C\,:\,\vert \operatorname{Arg}(x)-\theta \vert <\alpha/2\}.
 \end{align*} is a complex sector centered along the direction $\theta\in (-\pi,\pi)$ and of opening $\alpha\in (0,2\pi)$ in the complex plane. The desired normal form transformation is then obtained through the Laplace transform
\begin{align*}
 \mathcal L_{\theta}(H)(x)= \int_{0}^{\theta \infty} H(w) e^{-w/x} dw.
\end{align*}
Here $\mathcal L_{\theta}(H)$ can be analytically continued to an analytic function on a local sector $B(\nu)\cap S(\theta,\alpha+\pi)$, with $B(\nu)$ being the open ball centered at the origin with radius $\nu>0$, see further details below.

The novelty of \cite{bonckaert2008a} lies in the introduction of a norm turning the space of analytic funcions of exponential growth (as above) into a Banach convolutional algebra. Similar norms have been used by different authors, see e.g. \cite{costin2009,sauzin2015}, but not in a systematic way for equations of the type \eqref{prenormalform}. We refer to \cite{de2020a} for an interesting application of the convolutional algebra in the context of slow-fast equations. 

In this paper, we consider the general case of \eqref{prenormalform} with $k\in \mathbb N$, which was left open by the authors of \cite{bonckaert2008a}. {At the same time, we will allow $A$ to have nontrivial Jordan blocks.} In contrast to the belief expressed by the authors of \cite{bonckaert2008a}, we show that this extension does not rely on multi-summability, see \cite{braaksma1992a}. Instead, we can simply work with a   generalized Laplace transform of order $k\in \mathbb N$:
\begin{align*}
 \mathcal L_{\theta,k}(H)(x): = \int_{0}^{\theta \infty} H(w) e^{-w^k/x^k} kdw.
\end{align*}
\rsp{For more general references on multi-summability and exact asymptotics in the context of complex linear systems, we refer to \cite{balser1991a,malgrange1992a,put2003a,sibuya1975linear}}.
\begin{remark}
    In contrast to the Laplace transforms in \cite{balser2000a} or \cite{braaksma1992a}, this version does not commute with the ramification operator $R_k\colon w\mapsto w^{1/k}$:
    $
        \mathcal L_{\theta,k}(H)\circ R_k\not= \mathcal L_{k\theta,1}(H\circ R_k).
    $
    Instead, we have
    \[
        \mathcal L_{\theta,k}(H)\circ R_k = \mathcal L_{k\theta,1}(\widetilde{H}\circ R_k)
    \]
    with $\widetilde{H}(w)=H(w).w^{1-k}$.  There are some advantages to using our version of the Laplace transform, but also some disadvantages: the multiplication with $w^{1-k}$ in the above commutation causes mild singular behaviour at the origin and as a result several properties proven in \cite{bonckaert2008a} and notions introduced there for $k=1$ cannot be directly transported to the general case.
\end{remark}

Notice that this definition of $\mathcal L_{\theta,k}$ agrees with the usual one for $k=1$ ($\mathcal L_{\theta,1}=\mathcal L_{\theta}$). Moreover, $\mathcal L_{\theta,k}$ is for each $k\in 
\mathbb N$ defined on the space analytic functions $H=H(w)$ that have at most exponential growth of order $k$,  i.e. functions $H$ for which there exists a $\mu^k>0$ and a $C>0$ such that 
\begin{align*}
 \vert H(w)\vert e^{-\mu^k \vert w\vert^k}\le C\quad \forall\, w\in S(\theta,\alpha),
\end{align*}
    see further details below. 
    
Note that we have
\begin{align}
 H(w)=w^n\quad \Longrightarrow \quad \mathcal L_{\theta,k}(H)(x) = {x^{n+1}} \Gamma\left(\frac{n+1}{k}\right),\eqlab{Lmon0}
\end{align}
which explains that formally, Laplace transforms of convergent series may produce divergent series, or vice-versa, the inverse Laplace operator (Borel transform) may take away divergent behaviour for the same reason.  Note that the shift in exponents ($w^n\mapsto x^{n+1}$) restricts the Borel transform to those series that are $O(x)$.  It is a small price to pay, especially considering the fact that the more traditional Laplace transform in \cite{braaksma1992a} only transforms to series that are $O(x^k)$, which would form a somewhat bigger nuisance.  The Laplace transform of \cite{balser2000a} avoids this problem as well, but seems to be more difficult w.r.t.~finding a suitable norm in which the convolution operator is bounded, see Remark~\ref{rm:balser} later.

{The vector field \eqref{prenormalform} for which we claim to provide a simpler form is an example of a so-called 1-resonant system: there is only one source of resonance (namely the zero eigenvalue in the $x$ variable).  The simpler form in our case has the intent to reduce the number of nonresonant terms in the expression; in a narrow interpretation of the concept ``normal form'', our results would not arrive at such a normal form except in the cases where \emph{all} nonresonant terms are being canceled.  See for example \cite{bonckaert-verstringe} for a similar situation where such simpler forms may converge to actual normal forms under extra diophantine conditions.  Likewise, under extra diophantine conditions, summable normal forms of \eqref{prenormalform} have been obtained in \cite{braaksma2007a} (and \cite{stolovitch1996a}).  Systems that are not 1-resonant are far more difficult to study, and only some cases have been dealt with; we refer for example to the 2-resonant vector field that occurs in the study of the Painlev\'e system, see \cite{bittmann2018a}.
For a more general overview on normal forms we refer to \cite{stolovitch2009a}.
}

The paper is organized as follows: In \secref{defn}, we first present our definitions and then subsequently state our main result \thmref{main1} on the normal form of \eqref{prenormalform}. In \secref{proof}, we prove \thmref{main1} by extending the procedure in \cite{bonckaert2008a} (for $k=1$) to the general case. We have tried to find a balance between self-containment and compactness, and often chose, for the sake of readability, to reprove some arguments rather than to merely call upon their similarity to those in \cite{bonckaert2008a}. The treatment of the non semi-simple case is the part where this paper diverges the most from \cite{bonckaert2008a}.  Note that by relaxing the conditions on $A$, our results also generalize the $k=1$ case.
In \secref{zeroHopf}, we study an application to a real analytic zero-Hopf. Finally, in \secref{discussion} we conclude the paper. Here we also point out a misleading interpretation of the results in \cite{bonckaert2008a} (stated as remark) and discuss potential future directions.
\medskip

{PDM acknowledges support from FWO project G0F1822N.}

\section{Definitions and statement of   the main result}\seclab{defn}

Let $\mathbb N_0 = \mathbb N\cup\{0\}$. In the following, we often consider $\mathbf j =(j_1,\ldots,j_n)\in \mathbb N_0^n$ and as usual define
\begin{align*}
 \vert \mathbf j\vert: =j_1+\cdots+j_n\ge 0,\quad \langle  \mathbf s, \mathbf m\rangle :=\sum_{i=1}^n s_i m_i\quad \forall\, \mathbf j\in \mathbb N_0^n,\,\forall\,  \mathbf s,\mathbf m\in \mathbb C^n.
\end{align*}
For $z\in \mathbb C^n$, we use the norm $\vert z\vert=\max\{\vert z_1\vert,\ldots,\vert z_n\vert\}$ and write 
\begin{align}
z^{\mathbf j}:=z_1^{j_1}\cdots z_n^{j_n}\quad \forall\, \mathbf j\in \mathbb N_0^n.\eqlab{zj}
\end{align} We also let $B^n(r)$ denote the ball in $\mathbb C^n$ centered at $z=0$ and of radius $r>0$. (We drop the superscript on $B^n$ in the case $n=1$.) Finally, we define $\mathcal A^{n,m}(r)$ as the set of analytic functions defined on $B^n(r)$ with values in $\mathbb C^m$.  We equip these spaces with the Banach norm:
\begin{align*}
 \Vert h\Vert := \sup_{z\in B^n(r)}\vert h(z)\vert\quad \forall\, h\in \mathcal A^{n,m}(r).
\end{align*}

Next, we consider formal series 
\begin{align}
 \widehat h(x,z) = \sum_{\ell=1}^\infty h_\ell(z)x^\ell\quad \mbox{with}\quad h_\ell\in \mathcal A^{n,m},\,\forall\, \ell\in \mathbb N_0;\eqlab{formalF}
\end{align}
notice that the sum starts with $\ell=1$ and consequently $\widehat h(0,\cdot)=0$. 
\begin{definition}
 $\widehat h(x,z)$ is Gevrey-$\frac{1}{k}$ with $k\in \mathbb N$ if there exists constants $K,T>0$ such that 
 \begin{align*}
  \Vert h_\ell\Vert\le KT^{\ell-1} \Gamma \left(\frac{\ell}{k}\right)\quad \forall\,\ell\in \mathbb N.
 \end{align*}
\end{definition}

In this paper, we define the Borel transform $\mathcal B_k$ (with respect to $x$, leaving $z$ fixed) of order $k$ as follows: 
Let $\widehat h$ be a formal series, see \eqref{formalF}. Then 
\begin{align}
 \mathcal B_k(\widehat h)(w,z): =\sum_{\ell=1}^\infty \frac{h_\ell(z)}{\Gamma\left(\frac{\ell}{k}\right)} w^{\ell-1}.\eqlab{borel}
\end{align}
It is clear that if $\widehat h$ is Gevrey-$\frac{1}{k}$ then $H:=\mathcal B_k(\widehat h)$ defines an absolutely convergent series in the disc $\vert w\vert<T^{-1}$:
\begin{align*}
 \vert H(w,z)\vert\le \sum_{\ell=1}^\infty \frac{\Vert h_\ell\Vert}{\Gamma\left(\frac{\ell}{k}\right)} w^{\ell-1}\le  \frac{K}{1-T\vert w\vert}.
\end{align*}
It follows that $H:B(T^{-1}) \times B^n(r)\rightarrow \mathbb C^m$ is an analytic function. 

Next, we turn to the definition of $k$-summability and the Laplace transform of order $k\in \mathbb N$: 
\begin{definition}\defnlab{ksum}
Suppose that $H=\mathcal B_k(\widehat h)$ can be analytically continued to $(w,y)\in S(\theta,\alpha)\times B^n(r)$, where $S(\theta,\alpha)\subset \mathbb C$ is the sector centered along the complex direction $\theta$ with opening $\alpha\in (0,\pi)$. Suppose additionally that $H$ has at \textit{most exponential growth of order $k\in\mathbb N$}, i.e. there is a $\mu>0$ and a constant $C>0$ such that 
\begin{align}
 \vert H(w,z)e^{-\mu^k \vert w\vert^k} \vert \le C\quad \forall\,w\in S(\theta,\alpha),\,z\in B^n(r).\eqlab{condHexpmu}
\end{align}
Then we say that $\widehat h$, see \eqref{formalF}, is \textnormal{$k$-summable in the direction $\theta$} and its \textnormal{Laplace transform of order $k$:}
\begin{align}
  \mathcal L_{\theta,k}(H)(x,z) = \int_{0}^{\theta \infty} H(w,z) e^{-w^k/x^k} kdw,\eqlab{laplacekHxy}
\end{align}
is its $k$-sum. 
\end{definition}
The integral in \eqref{laplacekHxy} can be analytically continued to \begin{align}\eqlab{omega0}x\in \omega_k(\nu,\theta,\alpha):=B(\nu)\cap S\left(\theta,\alpha+\pi/k\right),\end{align} for $\nu>0$ small enough, see \figref{omega}(b), \lemmaref{laplace} and \lemmaref{hxy} below for further details. 

We will define $k$-sums of $H:S(\theta,\alpha)\rightarrow \mathbb C$ (that are independent of $z$) and $H=(H_1,\ldots,H_n):S(\theta,\alpha)\times B^n(r)\rightarrow \mathbb  C^n$ completely analogously (in the latter case, when \defnref{ksum} applies to each component). 

\subsection{Main result}
We will describe our main result in terms of \eqref{prenormalform} under the following assumption: 
\begin{assumption}\assumplab{Adiag}
$A$ is in Jordan normal form:
\begin{align}
A = \Lambda +\Xi,\eqlab{Adiag}
\end{align} where 
\begin{enumerate}
\item $\Lambda  = \operatorname{diag}(\lambda_1,\ldots,\lambda_n)$, with all $\lambda_i$ being nonzero
\item $\Xi$ vanishes for $n=1$ and is nilpotent otherwise with all nonzero elements on the super-diagonal:
\begin{align*}
 (\Xi)_{i,j} =: \begin{cases}
                0 & \text{if}\quad j\ne i+1,\\
                \xi_{i} & \text{if}\quad j=i+1,
               \end{cases}\quad i,j\in \{1,\ldots,n\}.
\end{align*}
Here
\begin{align}
 \xi_{i}\in \{0,1\}\quad \forall\,i\in \{1,\ldots,n-1\}\quad \mbox{and}\quad \left(\xi_i=1, \,i\in \{1,\ldots,n-1\}\quad \Longrightarrow \quad \lambda_{i+1}=\lambda_i\right) .\eqlab{xiicond}
\end{align}
\end{enumerate}
\end{assumption}
The remaining assumptions of the following result are similar to those of \cite[Theorem 3]{bonckaert2008a}.
In particular, \eqref{condlambdak} replaces the condition stated after \cite[Eq. (9)]{bonckaert2008a}.
\begin{theorem} \thmlab{main1}
Consider \eqref{prenormalform}, satisfying \assumpref{Adiag} with $f$ being analytic near the origin, \rsp{and any $C>0$}.
\rsp{We then define $\mathcal R=\mathcal R(C)$ by }
  \begin{align}
\rsp{\mathcal R: = \left\{(i,\mathbf j)\in \{1,\ldots,n\}\times \mathbb N_0^n\,:\,\vert \lambda_i - \langle \mathbf j,\lambda\rangle \vert \le C(1+\vert \mathbf j\vert)\right\}.}
\eqlab{mathR0}
  \end{align}
  Moreover, let $\theta$ be a complex direction, so that the following holds for some $\xi>0$ small enough:
  \begin{align}
   \lambda_i -\langle \mathbf j,\lambda\rangle \notin S(k\theta,\xi)\quad \forall\,(i,\mathbf j)\,{\notin} \,\mathcal R,\eqlab{condlambdak}
  \end{align}
  Then there exist $R>0$, $\nu>0$ and an $\alpha>0$, all small enough, {and} two analytic functions $\phi=\phi(x,z)$ and $g=g(x,z)$ {with the following properties}:
  \begin{enumerate}
  \item[(i)]
  {$\phi$ and $g$} are $k$-sums of formal series $\widehat \phi$ and $\widehat g$, respectively;
  \item[(ii)] {$\phi$ and $g$}  are defined on $x\in \omega_k(\nu,\theta,\alpha)$, $z\in B^n(R)$ {(see \eqref{omega0})};
  \item[(iii)] {$\phi$ and $g$} extend continuously to the closure of the domain in (ii), 
\end{enumerate}
such that the following holds true:
  \begin{enumerate}
   \item ${g=(g_1,\ldots,g_n)}$ has a convergent power series expansion defined by $$ g_i(x,z)= \sum_{\mathbf j\,:\,(i,\mathbf j)\in \mathcal R}  g_{i,\mathbf j}(x) z^{\mathbf j},\quad x\in \omega_k(\nu,\theta,\alpha), \,z\in B^n(R),$$
  {with $i\in\{1,\ldots,n\}$}.
  \item The $x$-fibered diffeomorphism $(x,z)\mapsto (x,y)$ defined by \begin{align}
 y = z+x\phi (x,z),\eqlab{yz}
\end{align}
locally conjugates \eqref{prenormalform} with 
\begin{align}
 \dot{x} = \frac{1}{k}x^{k+1},\qquad\dot{z}= A z+ xg(x,z).\eqlab{finalnormalform}
\end{align}
\end{enumerate}
\end{theorem}
\begin{remark}\remlab{mathR0}
 \rsp{Notice that there is clearly a $C>0$ large enough such that 
 \begin{align*}
     \vert \lambda_i - \langle \mathbf j,\lambda\rangle \vert \le C(1+\vert \mathbf j\vert)\quad \forall\,(i,\textbf j)\in \{1,\ldots,n\}\times \mathbb N_0^n,
 \end{align*}
and therefore
 \begin{align*}
  \mathcal R = \{1,\ldots,n\}\times \mathbb N_0^n.
  \end{align*}
 However, in this case no terms are removed in \eqref{finalnormalform}. In practice, one is interested in applying the result with the "smallest possible" resonant set $\mathcal R$ (and a small $C$). }
\end{remark}
\begin{remark}
    {For $n=1$ with $A=\lambda\in \mathbb R\setminus \{0\}$, then 
    \begin{align*}
        \vert (j-1)\lambda\vert \ge \frac{1}{3}\vert \lambda\vert (1+j)\quad \forall\,j\in \mathbb N_0\setminus \{1\}.
    \end{align*}
    Consequently, we can take $\mathcal R =\{1\}\times \{1\}$ (corresponding to $j=1$) and $C=\frac13$ and for any $k\theta \notin \mathbb Z$, we then obtain the following normal form
    \begin{align*}
        \dot x &= \frac1k x^{k+1},\qquad  \dot z= (\lambda +x g_1(x))z.
    \end{align*}
It is subsequently possible to achieve $g_1(x)=\mu_{k}x^{k-1}$ if we allow division of the right hand side by a nonzero function, in agreement with \cite{martinet1983a}. 
This leads to the same normal form for saddle-nodes in dimension two by Martinet and Ramis \cite{martinet1983a}, upon which their analytic classification was based.}    \end{remark}

\section{Proof of \thmref{main1} } \seclab{proof}
To prove \thmref{main1}, we solve the conjugation equation for $g=g(x,z)$ and $\phi=\phi(x,z)$:
\begin{align}
  g - A\phi + \frac{\partial}{\partial z}\phi. Az  +\frac1k x^{k+1} \frac{\partial}{\partial x} \phi= f(x,z+x\phi )-x\frac{\partial}{\partial z}\phi .g-\frac1k x^k \phi ,
 \eqlab{conjeqn}
\end{align}
by setting up (following \cite{bonckaert2008a}) an associated equation for the Borel transforms (with respect to $x$, leaving $z$ fixed) $\Phi$ and $G$ of $\phi$ and $g$, respectively. (Notice that to obtain \eqref{conjeqn}, we simply differentiate \eqref{yz} with respect to $x$, use \eqref{prenormalform} and \eqref{finalnormalform} and rearrange.)
It is important to notice that \eqref{conjeqn} for $x=0$ reduces to:
\begin{align*}
g(0,z)-A\phi(0,z)+ \frac{\partial}{\partial z}\phi(0,z) Az = f(0,z),
\end{align*}
which
can be solved in the usual way by writing $g(0,z)=(g_{01}(z),\ldots,g_{0n}(z))$ and $\phi(0,z)$ as convergent power series in $z\in \mathbb C^n$, see also \cite[Lemma 10]{bonckaert2008a}; here $g_{0i}(z)$ only contain monomials $g_{0i\mathbf j}z^{\mathbf j}$ with $(i,\mathbf j)\in \mathcal R$ in agreement with the statement of \thmref{main1}. In this way, by performing a $z$-dependent shift of $g$ and $\phi$, we may in fact assume that $g(0,\cdot)=0$ and $\phi(0,\cdot)=0$ in \eqref{conjeqn}, see \cite[Lemma 10]{bonckaert2008a}, so that these functions do indeed have well-defined Borel transforms $G$ and $\Phi$. (Recall \eqref{borel}.)

To deal with the case where $\Xi\ne 0$, see \assumpref{Adiag}, $n\ge 2$, it will be important for us that
\begin{align}
    A=\Lambda +r\Xi,\eqlab{Ar}
\end{align}
for $r>0$. This form is obtained from \eqref{Adiag} by applying the scaling $z\mapsto (z_1,rz_2,\ldots,r^{n-1}z_n)$. We will prove \thmref{main1} by working in these coordinates with $r>0$ fixed small enough; this clearly does not change the statement of the theorem. It is important to note that $r>0$ will be fixed independently of the separate parameter $\mu>0$ that we introduce below.  

We present the equation for $\Phi$ and $G$ in Section \ref{conjbp} but first we extend some lemmas from \cite{bonckaert2008a} to the general case $k$. \thmref{main1} in the semi-simple case then follows the $k=1$ case of \cite{bonckaert2008a}. The details of the non-semi-simple case $\Xi\ne 0$ is treated in a separate section (section \ref{sec:nss}).

\subsection{The Borel plane and Laplace transforms}

Let
 \begin{align}
  \Omega = \Omega(\nu,\theta,\alpha) := B(\nu)\cup S(\theta,\alpha),\quad \nu>0,\,\alpha>0,\,\theta\in (0,2\pi),\eqlab{Omega}
 \end{align}
 see \figref{omega}.
We then first consider the space $\mathcal G_k=\mathcal G_k(\nu,\theta,\alpha,\mu)$ of analytic functions  $H:\Omega \rightarrow \mathbb C^n$ and equip this space with
the $\mu$-dependent Banach norm
\begin{align}
 \Vert H\Vert_{\mu,k}:=\sup_{w\in \Omega}\vert H(w)\vert \left(1+\mu^{2k} \vert w\vert^{2k}\right) e^{-\mu^k \vert w\vert^k},\quad k\in \mathbb N,\,\mu>0.\eqlab{Hnorm}
\end{align}

\begin{remark}
    Notice that in the case $k=1$, this norm agrees with the one in \cite[p. 312]{bonckaert2008a}.  
    One could define the norm for
general $k\in \mathbb N$ using the ramification operator $R_k(w) = w^{1/k}$. Notice in this regard that  \begin{align*}
   \dot x  =x^2\quad \Longleftarrow \quad \dot R_k(x) = \frac{1}{k}R_k(x)^{k+1}.
\end{align*}
But still the general case $k\in \mathbb N$ does not follow from \cite{bonckaert2008a}, since the set of functions $\{H \circ R_k : H \in \mathcal G_k\}$ is larger than $\mathcal G_1$. 
\end{remark}

In the following, we will take $\nu>0$ and $\alpha>0$ small and $\mu$ large enough. For this purpose, it is important to note that the following holds:
\begin{align*}
 \mathcal G_k(\nu,\theta,\alpha,\mu')\subset \mathcal G_k(\nu,\theta,\alpha,\mu)\quad \mbox{and}\quad \Vert H\Vert_{\mu',k}\le \Vert H\Vert_{\mu,k}\quad \forall\,\mu'\ge \mu,\,H\in \mathcal G_k(\nu,\theta,\alpha,\mu).
\end{align*}
This just follows from the monotonicity of $Q(s) =\left(1+s^2 \vert w\vert^{2k}\right) e^{-s \vert w\vert^k}$:
\begin{align*}
 Q'(s)=-\vert w\vert^k (1-s \vert w\vert^k)^2  e^{-s \vert w\vert^k}\le 0\quad \forall\, s\ge 0,\,\vert w\vert\ge 0.
\end{align*}

\begin{figure}[ht]
\begin{center}
\subfigure[$w\in \mathbb C$]{\includegraphics[width=.43\textwidth]{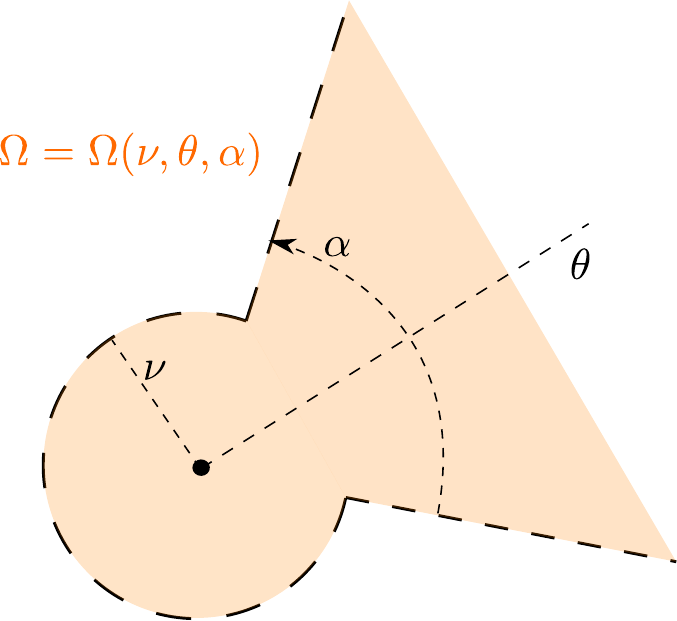}}
\subfigure[$x\in \mathbb C$]{\includegraphics[width=.45\textwidth]{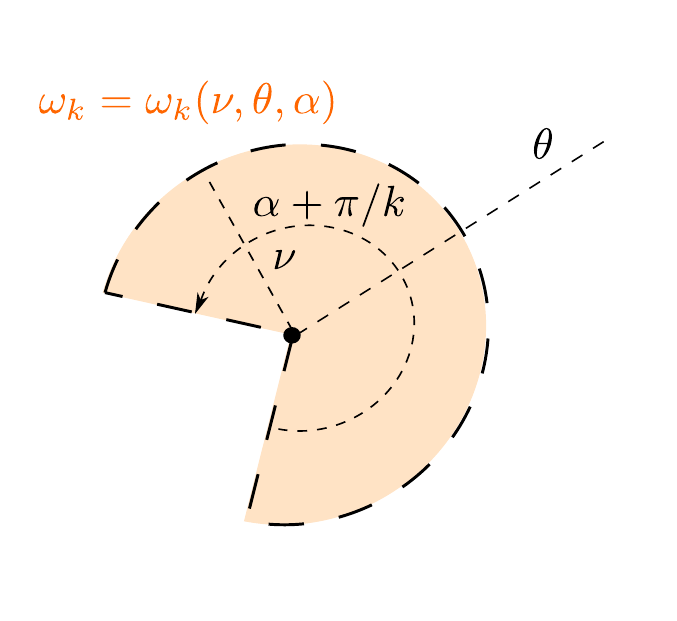}}
\end{center}
\caption{Illustrations of the different domains $\Omega$ and $\omega_k$. Fig. (a) is in the ``Borel plane'' $w\in \mathbb C$, whereas Fig. (b) is in the $x$-domain. Here $\omega_k$ is the complex domain upon which the Laplace transform of $\mathcal G_k$-functions on $\Omega$ are defined, see \lemmaref{laplace}. }
\figlab{omega}
\end{figure}

\begin{lemma}\lemmalab{laplace}
$\mathcal L_{\theta,k}$, given by
\begin{align}
 \mathcal L_{\theta,k}(H)(x) = \int_{0}^{\theta \infty} H(w) e^{-w^k/x^k} kdw,\eqlab{laplacekH}
\end{align}
defines a bounded linear operator from $\mathcal G_k$ to the space of analytic functions defined on the local sector: 
\begin{align*}
 \omega_k(\nu,\theta,\alpha):= B(\nu)\cap S(\theta,\alpha+\pi/k),
\end{align*}
recall \eqref{omega0},
for $0<\nu<\mu^{-1} \sqrt[k]{\sin \frac{k\alpha}{4}}$ small enough. In particular, the operator norm satisfies:
\begin{align*}
 \Vert \mathcal L_{\theta,k}\Vert \le \Gamma\left(\frac{1}{k}\right) \frac{\nu}{\sqrt[k]{ \sin \frac{k\alpha}{4}-\mu^k \nu^k}}.
\end{align*}

\end{lemma}
\begin{proof}
 The proof follows \cite[Proposition 3]{bonckaert2008a}: Given $x\in \omega_k$, so that $x\in S(\theta,\alpha+\pi/k)$ and $\vert x\vert \le \nu$, see \figref{omega}, it is clear that there is a $\tilde \theta \in (\theta-\alpha/2,\theta+\alpha/2)$ so that $x\in S(\tilde \theta,\pi/k-\alpha/2)$. Then
 \begin{align*}
  \vert (\mathcal L_{\tilde \theta,k})(H))(x)\vert &\le \Vert H\Vert_{\mu,k}\int_0^{\tilde \theta \infty} e^{\mu^k \vert w\vert^k} e^{-\text{Re}(w^k/x^k)}k \vert dw\vert\\
  &\le \Vert H\Vert_{\mu,k}\int_0^{\infty} e^{\mu^k t^k} e^{-t^k/\vert x\vert^k \cos( k(\tilde \theta -\operatorname{Arg}(x))) } k dt
 \end{align*}
 We have $\vert\tilde \theta -\operatorname{Arg}(x)\vert <\frac{\pi}{2k}-\frac{\alpha}{4}$ and therefore 
\begin{align*}
 \cos( k(\tilde \theta -\operatorname{Arg}(x))) \ge \sin \frac{k\alpha}{4}. 
\end{align*}
Consequently, for any $\vert x\vert <\nu< \mu^{-1} \sqrt[k]{\sin \frac{k\alpha}{4}}$ we have
\begin{align*}
 \int_0^{\infty} e^{\mu t^k} e^{-t^k/\vert x\vert^k \cos( k(\tilde \theta -\operatorname{Arg}(x))) } k dt&\le \int_0^{\infty} e^{\mu^k t^k} e^{-t^k/  \vert x\vert^k \sin\frac{k\alpha}{4} } kdt\\
 &= \Gamma\left(\frac{1}{k}\right) \frac{\vert x\vert}{\sqrt[k]{\sin\frac{k\alpha}{4}-\mu^k \vert x\vert^k}}\\
 &= \Gamma\left(\frac{1}{k}\right) \frac{\nu}{\sqrt[k]{\sin\frac{k\alpha}{4}-\mu^k \nu^k}}.
\end{align*}
This proves the statement since $\theta\rightarrow \tilde \theta$ corresponds to a deformation of the complex path integral (used to analytically continue $\mathcal L_{\theta,k}(H)$).
\end{proof}

We now define the convolution operator of order $k$ as
 \begin{align}
 (H\star_k J)(w) := w \int_0^{1} \mathbb K(s)H(w(1-s)^{1/k}) J(ws^{1/k}) ds,\eqlab{HJconv2}
 \end{align}
with the kernel \[\mathbb K(s):= (s(1-s))^{\frac{1-k}{k}},\]

\begin{remark}
    For $k=1$, it is the traditional convolution which is known to work well with the Laplace transform of order $k=1$: the convolution of two functions transforms to the product of the transforms of the two functions.  The expression above arises after imposing the following commutation rule at order $k$:
    \[
        \widetilde{(H \star_k J)}\circ R_k = (\widetilde{H} \circ R_k) \star_1 (\widetilde{J} \circ R_k)
    \]
    where $\tilde{\cdot}$ means a multiplication with a function $w\mapsto w^{1-k}$ and where $R_k$ is the ramification operator $R_k(w)=w^{1/k}$.
\end{remark}

\begin{lemma}\lemmalab{conv0}
For any $H,J\in \mathcal G_k$, we have that
 \begin{align}
  \mathcal L_{\theta,k} (H\star_k J) = (\mathcal L_{\theta,k}H) (\mathcal L_{\theta,k}J),\eqlab{product}
 \end{align}
and 
\begin{align}
 \frac{1}{k} x^{k+1}\frac{d}{dx}\mathcal L_{\theta,k}(H)(x) = \mathcal L_{\theta,k}((\cdot )^kH)(x),\eqlab{important}
\end{align}
with $(\cdot )^kH$ denoting the function $w\mapsto w^kH(w)$. 
\end{lemma}
 \begin{proof}
  The proof of \eqref{important} follows from a straightforward calculation, using the absolute convergence of $\mathcal L_{\theta,k}$. For \eqref{product}, we proceed as in the standard proof for the classical convolution: Firstly, it is without loss of generality to take $\theta=0$. We then consider the right hand side 
  \begin{align*}
   (\mathcal L_{\theta,k}H)(x) (\mathcal L_{\theta,k}J)(x) &=  \lim_{L\rightarrow \infty} \int \int_{0\le w,u\le L}k^2  H(w)J(u)e^{-(w^k+u^k)/x^k} dwdu\\
   &=\lim_{L\rightarrow \infty} \int \int_{w,u\ge 0,\,0\le w+u\le L} k^2 H(w)J(u)e^{-(w^k+u^k)/x^k} dwdu,
  \end{align*}
 using that $H$ and $J$ belong to $\mathcal G_k$ in the second equality. We then use the substitution:
 \begin{align*}
  w = p(1-s)^{1/k},\quad u=p s^{1/k},
 \end{align*}
so that $w^k+u^k = p^k$. The Jacobian of this transformation is given by 
\begin{align*}
\frac{1}{k} p\mathbb K(s).
\end{align*}
Consequently, \begin{align*}
   (\mathcal L_{\theta,k}H)(x) (\mathcal L_{\theta,k}J)(x) 
   &=\lim_{L\rightarrow \infty} \int_0^L k e^{-p^k/x^k} (H\star_k J)(p) dp,
  \end{align*}
  as desired.
 \end{proof}

\begin{lemma} \lemmalab{conv}
The convolution $\star_k$ given by \eqref{HJconv2} defines a continuous bilinear operator $\mathcal G_k\times \mathcal G_k\rightarrow \mathcal G_k$ for all $\mu>0$, $k\in \mathbb N$. In particular, for each $k\in \mathbb N$ define 
 \begin{align*}
 q_k := \rspp{\frac{2^{\frac{4k-1}{k}}  \pi}{2\cos\left(\frac{\pi(k-1) }{2k}\right) }}.
 \end{align*}
 Then $q_k>0$ and for any $H,J\in \mathcal G_k$, we have the following estimate
 \begin{align*}
  \Vert H\star_k J\Vert_{\mu,k} \le \frac{q_k}{\mu}\Vert H\Vert_{\mu,k} \Vert J\Vert_{\mu,k}.
 \end{align*} 
\end{lemma}
\begin{proof}
We have
 \begin{align*}
  \vert (H\star_k J)(w)\vert &\le \Vert H\Vert_{\mu,k} \Vert J\Vert_{\mu,k}  \int_0^{1} \mathbb K(s) \frac{\vert w\vert e^{\mu^k \vert w\vert^k}}{(1+\mu^{2k} \vert w\vert^{2k}(1-s)^2)(1+\mu^{2k} \vert w\vert^{2k}s^2)}ds\\
  & \le  \Vert H\Vert_{\mu,k} \Vert J\Vert_{\mu,k} \frac{ e^{\mu^k \vert w\vert^k}}{1+\mu^{2k} \vert w\vert^{2k}}  \left[8\int_0^{\frac12} \mathbb K(s) \frac{\vert w\vert}{1+\mu^{2k} \vert w\vert^{2k}s^2}ds\right].
 \end{align*}
 In the second inequality, we have used a symmetry argument, based upon 
 \begin{align*}
  \mathbb K(s)=\mathbb K(1-s)\quad \forall\, s\in [0,1],
 \end{align*}
 and the fact that $$1+\mu^{2k}\vert w\vert^{2k} (1-s)^2 > \frac14 (1+\mu^{2k} \vert w\vert^{2k})\quad \forall\, s\in [\frac12,1],$$ following \cite[Proposition 4]{bonckaert2008a}. We now estimate the remaining integral. For this we first use that $\mathbb K(s)\le 2^{\frac{k-1}{k}} s^{\frac{1-k}{k}}$ for $s\in [0,\frac12]$ so that
 \begin{align*}
 \int_0^{\frac12} \mathbb K(s) \frac{\vert w\vert}{1+\mu^{2k} \vert w\vert^{2k}s^2}ds &\le 2^{\frac{k-1}{k}}  \int_0^{\frac12} s^{\frac{1-k}{k}} \frac{\vert w\vert }{1+\mu^{2k} \vert w\vert^{2k}s^2}ds\\
 &\rspp{\le \frac{2^{\frac{k-1}{k}} }{\mu} \int_0^\infty \frac{u^{\frac{1-k}{k}} }{1+u^2}du}\\
 &\rspp{= \frac{2^{\frac{k-1}{k}} }{\mu}\frac{\pi}{2\cos\left(\frac{\pi(k-1) }{2k}\right) }}.
\end{align*}
\rspp{Here we have also used the substitution $u=\mu^k \vert w\vert^k s$ in the second inequality.}
%
Combining these estimates, we find that
\begin{align*}
 \Vert H\star_k J\Vert_{\mu,k} &\le \Vert H\Vert_{\mu,k}\Vert J \Vert_{\mu,k}\left[8\int_0^{\frac12} \mathbb K(s) \frac{\vert w\vert}{1+\mu^{2k} \vert w\vert^{2k}s^2}ds\right] \\
 &\le  \Vert H\Vert_{\mu,k}\Vert J \Vert_{\mu,k} \left[\rspp{\frac{1}{\mu}\frac{ 2^{\frac{4k-1}{k}}  \pi}{2\cos\left(\frac{\pi(k-1) }{2k}\right) }}\right]\\
 &:=\Vert H\Vert_{\mu,k}\Vert J \Vert_{\mu,k}\frac{q_k}{\mu},
\end{align*}
as desired.
\end{proof}
\begin{remark}
    For $k=1$ we have \rspp{$q_1=4\pi$ in agreement with the estimate in} \cite[Proposition 4]{bonckaert2008a}.
\end{remark}
\begin{remark}\label{rm:balser}
    It is for having a convenient proof of the bilinearity of the convolution operator that we have preferred not to work with the Laplace transform of \cite{balser2000a}: there the convolution is a bit more involved, even at order 1:
    \[
    (H \star J)(w) = \frac{d}{dw}\int_0^w H(s)J(w-s)ds.
    \]
    The need for computing a derivative is an extra obstruction in proving boundedness.
\end{remark}

We now turn to the Borel transform. First we need an auxiliary result.
\begin{lemma}\lemmalab{funch}
The function $H:\mathbb C\rightarrow \mathbb C$ defined by $$H(w) = \sum_{m=1}^\infty \frac{w^{m-1}}{\Gamma\left(\frac{m}{k}\right)}, \quad w\in \mathbb C,$$ is entire and belongs to $\mathcal G_k$ for any $\mu> 1$: 
\begin{align*}
 \Vert H\Vert_{\mu,k}\le C_k,\quad C_k>0;
\end{align*}
specifically, if $\mu>\sqrt[k]{2}$ then we have the following explicit form of the constant $C_k$:
\begin{align}
 \frac{C_k}{k}=384  e^{-2 - \frac{\sqrt{15}}{2}}(4 + \sqrt{15})^2 \approx 465.%
\eqlab{Cmu}
\end{align}
 \end{lemma}
 \begin{proof} 
We write 
\begin{align}
H(w) = \sum_{m=1}^{2k-1} \frac{w^{m-1}}{\Gamma\left(\frac{m}{k}\right)}+\sum_{m=2k}^\infty \frac{w^{m-1}}{\Gamma\left(\frac{m}{k}\right)}.\eqlab{Hsplit}
\end{align}
The finite sum can be estimated as
\begin{align*}
\sum_{m=1}^{2k-1} \frac{\vert w\vert^{m-1}}{\Gamma\left(\frac{m}{k}\right)} \le 2 (2k-1)\max_{m\in \{1,\ldots, 2k+1\}} \vert w\vert^{m-1} < 4k(1+\vert w\vert^{2k}), 
\end{align*}
since $\frac12 < \text{min}_{t>0} \Gamma(t)\approx 0.9$ the minimum of the $\Gamma$-function. 
For the second term, 
we will use that $\Gamma(t)$ is increasing for $t\ge 2$: 
  \begin{align*}
   \Gamma\left(\frac{k\ell+m}{k}\right) \ge \Gamma\left(\ell+1\right) =\ell!
  \end{align*}
    for any $m\ge k$, $\ell\in \mathbb N$. For $\vert w\vert\le \frac{1}{2}$, the second term in \eqref{Hsplit} is clearly bounded by the geometric series $\sum_{m=2k}^{\infty} 2^{1-m}=2^{2(1-k)}\le 1$.
Next, for $\vert w\vert \ge \frac{1}{2}$, we estimate:
  \begin{align}
  \left| \sum_{m=2k}^\infty \frac{w^{m-1}}{\Gamma\left(\frac{m}{k}\right)}\right| &\le \sum_{\ell=2}^\infty \sum_{m=0}^{k-1} \frac{\vert w\vert^{k\ell+m-1}}{\Gamma\left(\frac{k\ell+m}{k}\right)}\nonumber\\
  &\le \sum_{\ell=0}^\infty \frac{\vert w\vert ^{k\ell}}{\ell!} \sum_{m=0}^{k-1} \vert w\vert^{m-1}\nonumber\\
  &\le  \sum_{\ell=0}^\infty \frac{\vert w\vert ^{k\ell}}{\ell!}  \vert w\vert^{-1} k\max_{m\in \{0,\ldots,2k\}}\vert w\vert^{m}\nonumber \\
    &\le 2 k e^{\vert w\vert^k}\left(1+\vert w\vert^{2k}\right),\eqlab{herest}
   \end{align}
   using in the last inequality that $\vert w\vert \ge \frac{1}{2}$. In turn, using the estimate for $\vert w\vert\le \frac{1}{2}$, we find that the estimate \eqref{herest} in fact holds for all $w\in \mathbb C$.  In this way, we obtain the following bound on $H$:
   \begin{align*}
    \vert H(w)\vert \le 6k(1+\vert w\vert^{2k})e^{\vert w\vert^k},
   \end{align*}
valid for all $w\in \mathbb C$. Therefore for all $\mu>1$
\begin{align*}
 \Vert H\Vert_{\mu,k}&\le 6k \sup_{r\ge 0} (1+r^2)e^{(1-\mu^k)r} \left(1+\mu^{2k}r^2\right)\\
 &\le 6 k \sup_{r\ge 0} e^{(1-\mu^k)r} \left(1+\mu^{2k} r^2\right)^2,
\end{align*}
which is clearly bounded. Specifically for $\frac12 \mu^k>1 \Leftrightarrow \mu >\sqrt[k]{2}$, a simple calculation gives the following
\begin{align*}
 e^{(1-\mu^k) r}\left(1+\mu^{2k} r^2\right)^2\le e^{-\frac12 \mu^k r}\left(1+\mu^{2k} r^2\right)^2 \le 64 e^{-2 - \frac{\sqrt{15}}{2}}(4 + \sqrt{15})^2,
\end{align*}
for all $r\ge 0$. Here we have used that 
\begin{align*}
 p(s):=e^{-\frac12 s}\left(1+s^2\right)^2\le \max_{s\ge 0}p(s) = p(s_0)\quad \forall\,s\ge 0 
\end{align*}
where $s_0:=4+\sqrt{15}$. This gives the desired result and \eqref{Cmu}.

 \end{proof}

\begin{lemma}\lemmalab{boundBk}
 Consider $$h(x)=\sum_{m=1}^\infty h_m x^m,$$
 with 
 $$\vert h_m\vert \le K T^{m-1} \quad \forall\,m\in \mathbb N,$$
 for some $K,T>0$, such that $h:B(T^{-1})\rightarrow \mathbb C$ is analytic. Then 
 \begin{align*}
 \mathcal B_k(h) = \sum_{m=0}^\infty \frac{h_m}{\Gamma\left(\frac{m}{k}\right)}(\cdot)^{m-1},
 \end{align*}
 belongs to $\mathcal G_k$ with norm
 \begin{align*}
  \Vert \mathcal B_k (h)\Vert_{\mu,k}\le  K C_k\quad \forall\, \mu>T,
 \end{align*}
 for some $C_k$;
 specifically, we may take $C_k>0$ to be given by \eqref{Cmu} whenever $\mu>\sqrt[k]{2} T$.
\end{lemma}
\begin{proof}
The result easily follows from \lemmaref{funch}. 
\end{proof}
Under the conditions of the lemma, we also have that 
\begin{align}
 \left(\mathcal L_{\theta,k}\circ \mathcal B_k\right) (h) = h.\eqlab{LB}
\end{align}
This follows from \eqref{laplacekH} and \eqref{borel} and the absolute convergence of $\mathcal L_{\theta,k}$. 

As in \cite{bonckaert2008a}, we define $\mathcal G_k^n$ completely analogously as the space of analytic functions $H:\Omega \rightarrow \mathbb C^n$, $H=(H_1,\ldots,H_n)$, with each component $H_l\in \mathcal G_k$, $\ell\in\{1,\ldots,n\}$. The following norm
\begin{align}
 \Vert H \Vert_{\mu,k}:=\sup_{\ell\in \{1,\ldots,n\}} \Vert H_l\Vert_{\mu,k},\eqlab{HnormGn}
\end{align}
is a Banach norm on $\mathcal G_k^n$ and all of the results above extend with little effort to these generalized spaces.

Let $H=(H_1,\ldots,H_n)\in \mathcal G_k^n$ and $\mathbf j\in \mathbb N_0^n\setminus \{\mathbf 0\}$. Then we will also define $H^{\star_k \mathbf j}$ in the obvious way:
\begin{align*}
 H^{\star_k \mathbf j} = \underbrace{H_1 \star_k \cdots \star_k H_1 }_{\mbox{$j_1$ times}}\star_k \cdots \star_k \underbrace{H_l \star_k \cdots \star_k H_l }_{\mbox{$j_\ell$ times}}\star_k \cdots \star_k \underbrace{H_n \star_k \cdots \star_k H_n }_{\mbox{$j_n$ times}}.
\end{align*}
In this way, we have
\begin{align*}
 \mathcal L_{\theta,k} \left(H^{\star_k \mathbf j} \right) = \mathcal L_{\theta,k}(H_1)^{j_1}\times \cdots \times \mathcal L_{\theta,k}(H_n)^{j_n} := \mathcal L_{\theta,k}(H)^{\mathbf j}, 
\end{align*}
as desired, see \lemmaref{laplace} and recall the notation introduced in \eqref{zj}.

\subsection{\texorpdfstring{Definition of the space $\mathcal G_k^n\{z\}$}{Definition of the space Gkn\{z\}}}
The space $\mathcal G_k^n\{z\}$ consists of uniformly convergent series in $z$ with coefficients in $\mathcal G_k^n$:
\begin{align*}
 H(w,z) = \sum_{\mathbf j\in \mathbb N_0^n} H_{\mathbf j}(w) z^{\mathbf j},\quad H_{\mathbf j} \in \mathcal G_k^n,\,\forall\,\mathbf j\in \mathbb N_0^n,
\end{align*}
and we equip this space with the Banach norm:
\begin{align*}
 \Vert H\Vert_{\mu,k} := \sum_{\mathbf j\in \mathbb N_0^n} \Vert H_{\mathbf j}\Vert_{\mu,k} \mu^{-\vert \mathbf j\vert}\quad \mbox{with $\Vert H_{\mathbf j}\Vert_{\mu,k}$ given by \eqref{HnormGn}},
\end{align*}
see \cite[Section 6.2]{bonckaert2008a}.
Analytic functions $h(x,z)$ with $h(0,\cdot)=0$ have Borel transforms (with respect to $x$, leaving $z$ fixed) that belong to $\mathcal G^n_{\mu,k}\{z\}$ for $\mu>0$ large enough. In particular, we have the following (which is a simple consequence of \lemmaref{boundBk}, see also \cite[Lemma 7]{bonckaert2008a}).
\begin{lemma}\lemmalab{hxy}
 Consider
 \begin{align*}
  h(x,z)  = \sum_{\mathbf j\in \mathbb N_0^n} \sum_{\ell=1}^\infty h_{\ell\mathbf j} x^\ell z^{\mathbf j},
 \end{align*}
a convergent series with coefficients $h_{\ell\mathbf j}\in \mathbb C^n$, $\ell\in \mathbb N$, $\mathbf j\in \mathbb N_0^n$, and suppose that 
\begin{align*}
 \vert h_{\ell\mathbf j}\vert \le K T^{\ell-1+\vert \mathbf j\vert}\quad \forall\,\ell\in \mathbb N, \mathbf j\in \mathbb N_0^n.
\end{align*}
Then the Borel transform of $h$ (with respect to $x$, leaving $z$ fixed)
\begin{align*}
 \mathcal B_k(h)(w,z):=\sum_{\mathbf j\in \mathbb N_0^n} \left(\sum_{\ell=1}^\infty \frac{h_{\ell\mathbf j}}{\Gamma\left(\frac{\ell}{k}\right) } w^{\ell-1}\right) z^{\mathbf j}
\end{align*}
belongs to $\mathcal G^{n}_{\mu,k}\{z\}$ for all $\mu>T$, and 
\begin{align*}
 \Vert \mathcal B_k(h)\Vert_{\mu,k} \le 2^n K C_k,
\end{align*}
for some $C_k>0$. Furthermore, $$(\mathcal L_{\theta,k} \circ \mathcal B_{k})(h)(x,z)=h(x,z) \quad \forall \, x\in \omega_k(\nu,\theta,\alpha),\,z\in B^n(\mu^{-1}),$$
recall \eqref{omega0}, where $0<\nu<\mu^{-1} \sqrt[k]{\sin \frac{k\alpha}{4}}$. 
\end{lemma}

The final lemma deals with the Borel transform of terms of the form $h(x,z+x\phi (x,z))$ which appear in \eqref{conjeqn}. 
\begin{lemma}\lemmalab{hxycompose}
Consider the same assumptions as in \lemmaref{hxy}. Fix any $\xi>0$ and define 
\begin{align*}
 h^\star(\Psi)(w,z):=\sum_{\mathbf j\in \mathbb N_0^n} \left(\sum_{\ell=1}^\infty \frac{h_{\ell\mathbf j}}{\Gamma\left(\frac{\ell}{k}\right) } w^{\ell-1}\right) \star_k (z+\Psi(w,z))^{\star_k \mathbf j},\quad \Psi\in \mathcal G_k^n\{z\}.
\end{align*}
Then the following holds: $h^\star(\Psi)$ is well-defined for all $\Vert \Psi\Vert_{\mu,k}\le \xi$ for all $\mu\gg 1$, satisfying $\Vert h^\star(\Psi)\Vert_{\mu,k}\le 2^n KC_k$ and 
\begin{align*}
 \mathcal L_{\theta,k}(h^\star(\Psi))(x,z) = h(x,z+\mathcal L_{\theta,k}(\Psi)(x,z)).
\end{align*}
\end{lemma}
\begin{proof}Given \lemmaref{conv0}, \lemmaref{conv} and \lemmaref{hxy}, the proof of \lemmaref{hxycompose} follows from \cite[Proposition 5]{bonckaert2008a} and further details are therefore left out.
\end{proof}

\subsection{The conjugacy equation in the Borel plane}\label{conjbp}
We are now in position to formulate \eqref{conjeqn} in the Borel plane for the Borel transforms $G$ and $\Phi$ of $g$ and $\phi$. For the left hand side, we use \lemmaref{conv0} (in particular \eqref{product} and \eqref{important}) and obtain
\begin{align}
 \operatorname{LHS}(G,\Phi)(w,z):= G(w,z) -A\Phi(w,z) + \frac{\partial }{\partial z}\Phi(w,z) A z + w^k \Phi(w,z),\eqlab{lhs}
\end{align}
whereas for the Borel transform of the right hand side of \eqref{conjeqn}, we use \lemmaref{hxycompose} for the first  term and \eqref{product} on the remaining terms:
\begin{align}
 \operatorname{RHS}(G,\Phi)(w,z):=f^\star(1\star_k \Phi)(w,z) -1\star_k \frac{\partial}{\partial z}\Phi(w,z)\star_k G(w,z)-\frac{1}{k}w^{k-1} \star_k \Phi(w,z).\eqlab{rhs}
\end{align}
We solve the associated equation $\operatorname{LHS}=\operatorname{RHS}$ as in \cite{bonckaert2008a} through a fixed-point argument, that envolves a particular inversion of $\operatorname{LHS}$: 

Consider the auxiliary problem
\begin{align}
 \operatorname{LHS}(G,\Phi)=H\in \mathcal G_k^n\{z\},\eqlab{aux}
\end{align}
with $\operatorname{LHS}$ given by \eqref{lhs} and $H(w,z)=\sum_{\mathbf j\in \mathbb N_0} H_{\mathbf j}(w) z^{\mathbf j}$, $H_{\mathbf j}=(H_{1\mathbf j},\ldots,H_{n\mathbf j})\in \mathcal G_k^n$ fixed. 
We write 
\begin{equation}\eqlab{GPhiexpand}
\begin{aligned}
G=\sum_{\mathbf j\in \mathbb N_0^n} G_{\mathbf j}(w)z^{\mathbf j}\in \mathcal G_k^n\{z\},\quad &\Phi= \sum_{\mathbf j\in \mathbb N_0^n} \Phi_{\mathbf j}(w)z^{\mathbf j}\in \mathcal G_k^n\{z\},\\
G_{\mathbf j}=(G_{1\mathbf j},\ldots,G_{n\mathbf j})\in \mathcal G_k^n,\quad &\Phi_{\mathbf j}=(\Phi_{1\mathbf j},\ldots,\Phi_{n\mathbf j})\in \mathcal G_k^n.
\end{aligned} 
\end{equation}
 Our particular inverse $\operatorname{LHS}^{-1}:H\mapsto (G,\Phi)$ is then defined by
\begin{align}
G_{i\mathbf j} = 0 \quad \forall\,(i,\mathbf j)\notin \mathcal R\quad \mbox{and}\quad \Phi_{i\mathbf j}=0\quad \forall\,(i,\mathbf j)\in \mathcal R.\eqlab{cond0}
\end{align}
For the following lemma, we will need a final Banach space:
Let $\mathcal G_{k,1}^n\{z\}\subset \mathcal G_k^n\{z\}$ denote the subspace consisting of functions $\Phi\in \mathcal G_k^n\{z\}$ such that also
\begin{align*}
 \frac{\partial}{\partial z_i} \Phi \in \mathcal G_k^n\{z\},\quad \forall\,i\in\{1,\ldots,n\},
\end{align*}
and equip this space with the Banach norm
\begin{align*}
 \Vert \Phi\Vert_{\mu,k,1}:=\max \left\{\Vert \Phi\Vert_{\mu,k},\mu^{-1} \Vert \frac{\partial}{\partial z_1} \Phi\Vert_{\mu,k},\ldots, \mu^{-1} \Vert \frac{\partial}{\partial z_n} \Phi\Vert_{\mu,k}\right\}.
\end{align*}
\begin{lemma}\lemmalab{LHSinv}
Under the assumptions of \thmref{main1}, there exists an $r>0$ small enough, recall \eqref{Ar}, such that the linear mapping
 \begin{align*}
  \operatorname{LHS}^{-1}\,:\,\mathcal G_k^n \{z\}\rightarrow \mathcal G_k^n \{z\} \times \mathcal G_{k,1}^n \{z\} ,\quad  H\mapsto (G,\Phi),
  \end{align*}
  defined by \eqref{aux} and \eqref{cond0}, is uniformly bounded with respect to $\mu>0$. 
\end{lemma}
We will prove \lemmaref{LHSinv} in sections \ref{sec:ss} and \ref{sec:nss} below. In particular, in section \ref{sec:ss} we consider the semi-simple case $\Xi=0$ (independent of $r$) and then in the subsequent section \ref{sec:nss} we consider $\Xi\ne 0$. 
Both cases rest upon the following fact: For $\nu>0$ small enough, there is a constant $K>0$ such that 
\begin{align}\eqlab{lemma3est}
\rspf{\frac{1}{\vert \lambda_i-\langle \mathbf j,\lambda\rangle-w^k\vert}} \le \frac{K}{1+\vert \mathbf j\vert},\quad \forall\,(i,\mathbf j)\notin \mathcal R,\,w\in \Omega(\nu,\theta,\alpha).
\end{align}
This follows directly from the assumptions of \thmref{main1}, in particular \eqref{condlambdak}. In fact, it corresponds directly to \cite[Lemma 3]{bonckaert2008a} (upon replacing their $w$ by our $w^k$ \rspf{and their right hand side of \cite[Eq. (9)]{bonckaert2008a} by our $C(1+\vert \mathbf j\vert)$ (in line with our assumption \eqref{condlambdak}))}.

In fairness, the semi-simple case also follows from the non-semi-simple case (it corresponds to $r=0$), but we prefer to separate the proof nonetheless. This is motivated by the fact that $\Xi=0$ is substantially easier and follows \cite{bonckaert2008a} closely. For $\Xi\ne 0$, we will need to introduce some new notation and some additional auxiliary results (from \cite{de2020a}).
\subsection{Completing the proof of \thmref{main1}}
{Given \lemmaref{LHSinv}, we can now complete the proof of \thmref{main1}: Fix $r>0$ small enough. Then the particular inverse $\operatorname{LHS}^{-1}:H\in \mathcal G_k^n\{z\}\mapsto (G,\Phi)\in \mathcal G_k^n\{z\}\times \mathcal G_{k,1}^n\{z\}$ in \lemmaref{LHSinv} is well-defined and uniformly bounded with respect to $\mu>0$}. This then leads to the fixed point equation
\begin{align*}
 (G,\Phi) = (\operatorname{LHS}^{-1}\circ \operatorname{RHS})(G,\Phi),
\end{align*}
for a solution $(G,\Phi)$ of $\operatorname{LHS}=\operatorname{RHS}$, 
recall the definition of $\operatorname{RHS}$ in \eqref{rhs}.
\begin{proposition}\proplab{final}
 There is an $\xi>0$ and a $\mu>0$ large enough such that $(G,\Phi)\mapsto (\operatorname{LHS}^{-1}\circ \operatorname{RHS})(G,\Phi)$ is a contraction on the closed subset of $\mathcal G_k^n\{z\}\times \mathcal G_{k,1}^n\{z\}$ defined by $\Vert G\Vert_{\mu,k}\le \xi, \Vert \Phi\Vert_{\mu,k,1}\le \xi$.
\end{proposition}
\begin{proof}
The result follows from \cite[Proposition 8]{bonckaert2008a} as each argument there now carries over to our general setting for any $k\in\mathbb N$. Indeed, all the necessary estimates rely on  \lemmaref{conv}, \lemmaref{hxycompose} and \lemmaref{LHSinv} which generalize Proposition 4, Proposition 6 and Proposition 7 in \cite{bonckaert2008a}. 
\end{proof}
\thmref{main1} follows, setting $g=\mathcal L_{k,\theta}(G)$ and $\phi=\mathcal L_{k,\theta}(\Phi)$, once we finish the proof of \lemmaref{LHSinv}.
\subsection{Proof of \lemmaref{LHSinv} in the semi-simple case}\label{sec:ss}
{Here we prove \lemmaref{LHSinv} in the semi-simple case:}
\begin{align}
    {A=\operatorname{diag}(\lambda_1,\ldots,\lambda_n)\,\, \mbox{where all $\lambda_i$'s are nonzero}.\eqlab{Adiag1}}
\end{align}
{Then a simple calculation shows that the equation \eqref{aux} with $\operatorname{LHS}$ defined by \eqref{lhs} is equivalent with}
\begin{align}\eqlab{auxj}
 G_{i\mathbf j}(w) -\left(\lambda_i  - \langle \mathbf j,\lambda\rangle-w^k\right) \Phi_{i\mathbf j}(w) = H_{i\mathbf j}(w)\quad \forall\,i\in\{1,\ldots,n\},\,\mathbf j\in \mathbb N_0^n.
\end{align}
Indeed, we just insert \eqref{GPhiexpand} into \eqref{lhs} and collect terms.
Our particular inverse $\operatorname{LHS}^{-1}:H\mapsto (G,\Phi)$ of $\operatorname{LHS}$, defined by \eqref{cond0}, is then simply given by
\begin{align}\eqlab{LHSinv1}
 G_{i\mathbf j}& = H_{i\mathbf j}\quad \mbox{and} \quad \Phi_{i\mathbf j}=0\quad \forall\,(i,\mathbf j)\in \mathcal R,%
\end{align}
and 
\begin{align}\eqlab{LHSinv2}
 G_{i\mathbf j}(w)& = 0\quad \mbox{and} \quad \Phi_{i\mathbf j}(w)=-\frac{H_{i\mathbf j}(w)}{\lambda_i  - \langle \mathbf j,\lambda\rangle-w^k}\quad \forall\,(i,\mathbf j)\notin \mathcal R,\,w\in \Omega.
%
\end{align}
$\operatorname{LHS}^{-1}$ has two components and the $G$-component is clearly a projection and therefore to prove \lemmaref{LHSinv} under the assumption \eqref{Adiag1}, we only have to treat the $\Phi$-component. 
  For this we use \eqref{lemma3est} and \eqref{LHSinv2}, to obtain
 \begin{align*}
  \Vert \Phi_i\Vert_{\mu,k} = \sum_{\mathbf j\,:\,(i,\mathbf j)\notin \mathcal R} \Vert \Phi_{i\mathbf{j}}\Vert_{\mu,k}\mu^{-\vert \mathbf j\vert}\le K \sum_{\mathbf j\,:\,(i,\mathbf j)\notin \mathcal R} \frac{\Vert H_{i\mathbf{j}}\Vert_{\mu,k}}{1+\vert \mathbf j\vert} \mu^{-\vert \mathbf j\vert}\quad\forall\, i \in \{1,\ldots,n\}.
 \end{align*}
In this way, we have that 
\begin{align*}
  \Vert \Phi\Vert_{\mu,k} \le K\Vert H\Vert_{\mu,k}.
\end{align*}
For the derivatives, let $q\in \{1,\ldots,n\}$ and denote by $\mathbf e_q$ the $n$-tuple with one single nonzero element $(\mathbf e_q)_{q}=1$ at the $q$'th position. Then we notice that 
\begin{align}\eqlab{Philzq}
 \frac{\partial}{\partial z_q}\Phi_i = \sum_{\mathbf j\,:\,(i,\mathbf j)\notin \mathcal R} \Phi_{i\mathbf j} j_q z^{\mathbf j-\mathbf e_q},
\end{align}
and upon estimating and using \eqref{LHSinv2} and \eqref{lemma3est}, we obtain 
\begin{align*}
\Vert \frac{\partial}{\partial z_q}\Phi_i\Vert_{\mu,k} \le K\sum_{\mathbf j\,:\,(i,\mathbf j)\notin \mathcal R} \frac{\Vert H_{i\mathbf j}\Vert_{\mu,k}}{1+\vert \mathbf j\vert} j_q \mu^{-\vert \mathbf j-\mathbf e_q\vert}\le K\mu \Vert H\Vert_{\mu,k}\quad \forall\,i\in\{1,\ldots,n\},\,\, q\in\{1,\ldots,n\}.
\end{align*}
Here we have also used the definition of the norm on $\mathcal G_{k}^n\{z\}$ and that $\frac{j_q}{1+\vert \mathbf j\vert}\le 1$ for all $\mathbf j\in \mathbb N_0$, $q\in\{1,\ldots,n\}$. Subsequently, using the definition of the norm on $\mathcal G_{k,1}^n \{z\}$, we complete the proof.

\subsection[Proof of \lemmaref{LHSinv} for nonzero Xi]{Proof of \lemmaref{LHSinv} for $\Xi\ne 0$}\label{sec:nss}
We now turn to the proof of \lemmaref{LHSinv} in the non-semi-simple case $\Xi\ne 0$. Recall that $C>0$ and that 
\begin{align}\eqlab{mathR}
\rsp{\mathcal R = \left\{(i,\mathbf j)\in \{1,\ldots,n\}\times \mathbb N_0^n\,:\,
 \vert \lambda_i-\langle \mathbf j,\lambda\rangle \vert {\le} C (1+\vert \mathbf  j\vert)\right\}}.
\end{align}
Recall also that $A$ (cf. \eqref{Ar} and the discussion around there) is in the (scaled) Jordan normal form:
\begin{align*}
    A=\operatorname{diag}(\lambda_1,\ldots,\lambda_n)+r\Xi.
\end{align*}
Let
\begin{align*}
 \chi_s (\mathbf j)&:=\begin{cases} 
                     0 & \text{if}\quad j_s=0\\
                     1 & \text{if}\quad j_s\ne 0
                    \end{cases}\quad \forall\,\mathbf j=(j_1,\ldots,j_n)\in \mathbb N_0^n,
                    \end{align*}
                    and 
                    \begin{align}
                    \mathbf d_s&: = (0,\ldots,0,\underset{\mbox{$s$th}}{1},-1,0,\ldots,0)\in \mathbb N_0^n,\eqlab{ds}
\end{align}
for all $s\in\{1,\ldots,n-1\}$.
{It is then a simple calculation to show that our particular solution $(G,\Phi)$ (satisfying \eqref{cond0}) of $\operatorname{LHS}(G,\Phi)=H$ with $H_{\mathbf j}=(H_{1\mathbf j},\ldots,H_{n\mathbf j})\in \mathcal G_k^n$ given, is equivalently determined by the following equations for the coefficients: For any $i\in \{1,\ldots,n-1\}$:
\begin{align}
G_{i\mathbf j} -(\lambda_i-\langle \mathbf j,\lambda\rangle -w^k) \Phi_{i\mathbf j} = H_{i\mathbf j}+r\xi_{i}\Phi_{(i+1)\mathbf j}-r\sum_{s=1}^{n-1} \chi_{s+1}(\mathbf j)\xi_{s} (j_s+1) \Phi_{i(\mathbf j+\mathbf d_s)}\quad \forall\,\mathbf j\in \mathbb N_0^n,\eqlab{GijPhiij}
\end{align}
and for $i=n$:
\begin{align}
G_{n\mathbf j} -(\lambda_n-\langle \mathbf j,\lambda\rangle -w^k) \Phi_{n\mathbf j} = H_{n\mathbf j}-r\sum_{s=1}^{n-1} \chi_{s+1}(\mathbf j)\xi_{s} (j_s+1)\Phi_{n(\mathbf j+\mathbf d_s)}\quad \forall\,\mathbf j\in \mathbb N_0^n .\eqlab{GnjPhinj}
\end{align}
Indeed (as in the semi-simple case), we just insert \eqref{GPhiexpand} into \eqref{lhs}, rearrange and collect terms of order $z^{\mathbf j}$.


\begin{lemma}\lemmalab{GijPhiij}
Consider \eqref{GijPhiij} and \eqref{GnjPhinj} and suppose \eqref{cond0}. Then:
\begin{align}
G_{i\mathbf j} &= H_{i\mathbf j}\quad \forall\,(i,\mathbf j)\in \mathcal R, \eqlab{GH}
\end{align}
and for all $(i,\mathbf j)\notin \mathcal R, \,i\in \{1,\ldots,n-1\}$:
\begin{align}
-(\lambda_i-\langle \mathbf j,\lambda\rangle -w^k) \Phi_{i\mathbf j} &= H_{i\mathbf j}+r\xi_{i}\Phi_{(i+1)\mathbf j}-r\sum_{s=1}^{n-1} \chi_{s+1}(\mathbf j)\xi_{s} (j_s+1) \Phi_{i(\mathbf j+\mathbf d_s)}.\eqlab{eqni}
\end{align}
Finally, for $i=n$ the following holds true
\begin{align}
-(\lambda_n-\langle \mathbf j,\lambda\rangle -w^k) \Phi_{n\mathbf j} &= H_{n\mathbf j}-r\sum_{s=1}^{n-1} \chi_{s+1}(\mathbf j)\xi_{s} (j_s+1)\Phi_{n(\mathbf j+\mathbf d_s)} \quad \forall\,\mathbf j\,:\,(n,\mathbf j)\notin \mathcal R.\eqlab{eqnn}
\end{align}
\end{lemma}
\begin{proof}
 We only have to show \eqref{GH}, since \eqref{eqni} and \eqref{eqnn} are just \eqref{GijPhiij} and \eqref{GnjPhinj}, respectively, with $G_{i\mathbf j}=0$ for $(i,\mathbf j)\notin \mathcal R$, see \eqref{cond0}. For $(i,\mathbf j)\in \mathcal R$, we have $\Phi_{i\mathbf j}=0$ (by \eqref{cond0}) and
 \begin{align}
 G_{i\mathbf j} & = H_{i\mathbf j}+r\xi_{i}\Phi_{(i+1)\mathbf j}-r\sum_{s=1}^{n-1} \chi_{s+1}(\mathbf j)\xi_{s} (j_s+1) \Phi_{i(\mathbf j+\mathbf d_s)}\quad \forall\,\rspp{\mathbf j\,:\,(i,\mathbf j)\in \mathcal R},\,i\in\{1,\ldots,n-1\},\eqlab{GijR}
 \end{align}
 and
 \begin{align}
 G_{n\mathbf j} & = H_{n\mathbf j}-r\sum_{s=1}^{n-1} \chi_{s+1}(\mathbf j)\xi_{s} (j_s+1) \Phi_{i(\mathbf j+\mathbf d_s)}\quad \forall\,\rspp{\mathbf j\,:\,(n,\mathbf j)\in \mathcal R},\eqlab{GnjR}
 \end{align}
 by \eqref{GijPhiij} and \eqref{GnjPhinj}, respectively. From \eqref{xiicond} \rsp{and \eqref{mathR}}, we have 
 \begin{align*}
  ((i,\mathbf j)\in \mathcal R \,\,\mbox{and}\,\, \xi_i=1)\quad \Longrightarrow\quad (i+1,\mathbf j)\in \mathcal R.
 \end{align*}
Consequently, by \eqref{cond0} we have that
\begin{align}\eqlab{pp}
\xi_{i}\Phi_{(i+1)\mathbf j}=0\quad \forall\,(i,\mathbf j)\in \mathcal R.
\end{align}
Similarly, if $\xi_s=1$ so that $\lambda_{s+1}=\lambda_s$  by \eqref{xiicond}, then 
\begin{align*}
 \langle \mathbf j+\mathbf d_s,\lambda\rangle = \langle \mathbf j,\lambda\rangle +\lambda_s-\lambda_{s+1}=\langle \mathbf j,\lambda\rangle,
\end{align*}
using the definition of $\mathbf d_s$, see \eqref{ds}. This means that 
\begin{align*}
 ((i,\mathbf j)\in \mathcal R\,\,\mbox{and}\,\, \xi_s =1)\quad \Longrightarrow \quad (i,\mathbf j+\mathbf d_s)\in \mathcal R,
\end{align*}
recall \eqref{mathR}. In turn, the sums in \eqref{GijR} and \eqref{GnjR} vanish and by \eqref{pp} equations \eqref{GijR} and \eqref{GnjR} reduce to \eqref{GH} as claimed.
\end{proof}}


{We now focus on solving \eqref{eqni} and \eqref{eqnn}. For this purpose, we adapt the approach of \cite[Lemma 3.19]{de2020a}. To present the results, we use some notation from \cite{de2020a}: Let $\mathbf j,\mathbf m\in \mathbb N_0^n$. Then  $\mathbf p=(p_1,\ldots,p_\ell)\in \mathbb N_0^{\ell}$ for some $\ell\in \mathbb N_0$ is called \textit{a path from $\mathbf m$ to $\mathbf j$} if 
\begin{align*}
 \mathbf j = \mathbf m - \sum_{s=1}^{\ell} \mathbf d_{p_{s}}.
\end{align*}
Given $\mathbf j\in \mathbb N_0$, then 
\begin{align*}
 \mathbf C(\mathbf j): = \left\{\mathbf m\in \mathbb N_0^n \,:\,\exists \ell \in \mathbb N_0,\,\exists p_1,\ldots,p_l\in \{1,\ldots,n-1\},\,\mathbf j = \mathbf m - \sum_{i=1}^\ell \mathbf d_{p_i}\right\},
\end{align*}
is the set of all $\mathbf m\in \mathbb N_0^n$ for which there exists a path from $\mathbf m$ to $\mathbf j$. Clearly,
\begin{align}\eqlab{vmvj}
 \mathbf m\in \mathbf C(\mathbf j)\Longrightarrow \vert \mathbf m\vert = \vert \mathbf j\vert,
\end{align}
where $\vert \mathbf j\vert=j_1+\ldots+j_n$.
Moreover,
\begin{align*}
\mathbf P(\mathbf m,\mathbf j):=\left\{(p_1,\ldots,p_{\ell(\mathbf m,\mathbf j)} )\in \{1,\ldots,n-1\}^{\ell(\mathbf m,\mathbf j) } \,:\,\mbox{$(p_1,\ldots,p_{\ell(\mathbf m,\mathbf j)})$ is a path from $\mathbf m$ to $\mathbf j$}\right\},
\end{align*}
is the set of all paths from $\mathbf m$ to $\mathbf j$; obviously $\mathbf P(\textbf m,\mathbf j)$ may be empty but it is always  well-defined. In particular, the length $\ell(\mathbf m,\mathbf j)$ is independent of the path, see \cite[Proposition 3.14]{de2020a}.}

{Finally, define $\mathcal T[{i,\mathbf j}]:\mathcal G_k\rightarrow \mathcal G_k$ and $\mathcal K[{i, s,\mathbf j}]:\mathcal G_k\rightarrow \mathcal G_k$ as follows 
\begin{align*}
 \mathcal T[{i,\mathbf j}](h)(w)&: = \frac{-h(w)}{\lambda_i-\langle \mathbf j,\lambda\rangle -w^k}\quad \forall\,h\in \mathcal G_k,\,(i,\mathbf j)\notin \mathcal R,
 \end{align*}
 and 
 \begin{align*}
 \mathcal K[{i, s,\mathbf j}] (h)(w)&:=-r\xi_{s}(j_s+1) \mathcal T[{i,\mathbf j}](h)\quad \forall\,h\in \mathcal G_k,\,i\in\{1,\ldots,n\},\,s\in \{1,\ldots,n-1\},\,\mathbf j\in \mathbb N_0^n,\,r>0.
\end{align*}
From \eqref{lemma3est}, we obtain that 
\begin{align}
\Vert \mathcal T[{i,\mathbf j}](h)\Vert_{\mu,k}\le \frac{K}{1+\vert \mathbf j\vert}\Vert h\Vert_{\mu,k}\quad \forall\,h\in \mathcal G_k,\,(i,\mathbf j)\notin \mathcal R.\eqlab{estT}
\end{align}
In turn, it follows that 
\begin{align}\eqlab{vertK}
 \Vert \mathcal K[{i, s,\mathbf j}](h)\Vert_{\mu,k}\le \frac{r K(j_s+1)}{1+\vert \mathbf j\vert}  \Vert h\Vert_{\mu,k}\quad \forall\, h\in \mathcal G_k,\,s\in\{1,\ldots,n-1\},\,r>0.
\end{align}
We then use $\mathcal T[{i,\mathbf j}]$ to write \eqref{eqni} and \eqref{eqnn} in the equivalent forms
\begin{align}
  \Phi_{i\mathbf j} = \mathcal T[{i,\mathbf j}]( H_{i\mathbf j}+r\xi_{i}\Phi_{(i+1)\mathbf j})-r\sum_{s=1}^{n-1} \chi_{s+1}(\mathbf j)\xi_{s} (j_s+1)\mathcal T[{i,\mathbf j}](\Phi_{i(\mathbf j+\mathbf d_s)}) \quad \forall\,\mathbf j\,:\,(i,\mathbf j)\notin \mathcal R,\eqlab{eqni2}
\end{align}
for $i\in\{1,\ldots,n-1\}$
and 
\begin{align}
 \Phi_{n\mathbf j} = \mathcal T[{n,\mathbf j}]( H_{n\mathbf j})-r\sum_{s=1}^{n-1} \chi_{s+1}(\mathbf j)\xi_{s} (j_s+1)\mathcal T[{n,\mathbf j}](\Phi_{n(\mathbf j+\mathbf d_s)}) \quad \forall\,(n,\mathbf j)\notin \mathcal R,\eqlab{eqnn2}
\end{align}
for $i=n$. 
We first solve \eqref{eqnn2} (using $\mathcal K[{n,s,\mathbf j}]$) and here is why: If $\Phi_{(i+1)\mathbf j}$ is known then \eqref{eqni2} takes the same form as \eqref{eqnn2} (using $H_{i\mathbf j}+r\xi_{i}\Phi_{(i+1)\mathbf j} \leftrightarrow H_{n\mathbf j}$, $\Phi_{i\mathbf j}\leftrightarrow \Phi_{n\mathbf j}$, $i\leftrightarrow n$). Therefore once we have solved \eqref{eqnn2}, we can solve for \eqref{eqni2} for any $i\in\{1,\ldots, n-1\}$ inductively. }

 {
Notice that \eqref{eqnn2} reduces to 
\begin{align}
\Phi_{n\mathbf j} = \mathcal T[{n,\mathbf j}]( H_{n\mathbf j})\quad \forall\,\mathbf j=(j_1,0,\ldots,0)\in \mathbb N_0\times \{0\}^{n-1},\eqlab{j11}
\end{align}
since $\chi_{s+1}(\rsp{\mathbf j})=0$, $s\in \{1,\ldots,n-1\}$, for all such $\mathbf j$, 
and that 
\begin{align}\eqlab{j12}
    \mathbf C(\mathbf j)=\emptyset \quad \forall\,\mathbf j=(j_1,0,\ldots,0)\in \mathbb N_0\times \{0\}^{n-1};
\end{align}
this latter fact follows from $(j_1,0,\ldots,0)+\textbf d_s\notin \mathbb N_0^n$ for all $s\in \{1,\ldots,n-1\}$, recall \eqref{ds}.}

{
\begin{lemma}\lemmalab{Phinsol}
 Consider \eqref{eqnn2} with $H_{n\mathbf j}\in \mathcal G_k$ for all $\mathbf j\in \mathbb N_0^n$.  
 Then 
 \begin{align}
  \Phi_{n\mathbf j} = \mathcal T[{n,\mathbf j}](H_{n\mathbf j})+\sum_{\mathbf m\in \mathbf C(\mathbf j)} \sum_{\mathbf p\in \mathbf P(\mathbf m,\mathbf j)}  \left(\prod_{u=1}^{\ell(\mathbf m,\mathbf j)} \mathcal K[{n, p_u,\mathbf m-\sum_{v=1}^u \mathbf d_{p_v} }] \circ \mathcal T[{n,\mathbf m}] \right)(H_{n\mathbf m}) \quad \forall \,\mathbf j\,:\,(n,\mathbf j)\notin \mathcal R,\eqlab{Phinsol}
 \end{align}
 where $$\prod_{u=1}^{\ell(\mathbf m,\mathbf j)} \mathcal K[{n, p_u,\mathbf m-\sum_{v=1}^u \mathbf d_{p_v} }],$$ is understood as the composition of the operators $\mathcal K[{n, p_u,\mathbf m-\sum_{v=1}^u \mathbf d_{p_v} }]$ (which commute since $\mathcal T[{i,\mathbf j}]$ is defined by multiplication) and 
 where the sums  in \eqref{Phinsol} are understood to vanish for all $\mathbf j=(j_1,0,\ldots,0)\in \mathbb N_0\times \{0\}^{n-1}$ (where the summation index sets are empty, recall \eqref{j12}). 
\end{lemma}
\begin{proof}
The proof is by induction, with $\mathbf j=(j_1,0,\ldots,0)$ being the base case (which is true by \eqref{j11}). 
The induction step then follows \cite[Lemma 3.19]{de2020a} and we will therefore only sketch the proof. It is based upon an ordering on $\mathbb N_0^n$ defined as follows: $\mathbf m<\mathbf j$ if the first nonzero value (starting from the right) of $\mathbf j-\mathbf m$ is positive. Clearly, $\mathbf j+\mathbf d_s<\mathbf j$ for all $s\in\{1,\ldots,n-1\}$ and all \rspp{$\mathbf j\in \mathbb N_0^n$, recall \eqref{ds}}, and the induction hypothesis therefore assumes that \eqref{Phinsol} holds true for $\Phi_{n (\mathbf j+\mathbf d_s)}$ for all $s\in\{1,\ldots,n-1\}$.  To complete the induction step, we have to show that \eqref{Phinsol} holds true for all $\Phi_{n \mathbf j}$. To proceed,  we then use \cite[p. 374]{de2020a} to rewrite the sum in the definition of $\Phi_{n\mathbf j}(w)$ so that
 \begin{align*}
  \Phi_{n\mathbf j}& = \mathcal T[{n,\mathbf j}](H_{n\mathbf j})+\sum_{s=1}^{n-1} \chi_{s+1}(\mathbf j) \mathcal K[{n,s,\mathbf j}] \circ\Bigg[ \mathcal T[{n,\mathbf j+\mathbf d_s}](H_{n(\mathbf j+\mathbf d_s)})\\
  &+ \sum_{\mathbf m\in \mathbf C(\mathbf j+\mathbf d_s)} \sum_{\mathbf p\in  \mathbf P(\mathbf m,\mathbf j+\mathbf d_s)} \Bigg(\prod_{u=1}^{\ell(\mathbf m,\mathbf j+\mathbf d_s)} \mathcal K[{n, p_u,\mathbf m-\sum_{v=1}^u \mathbf d_{p_v} }] \circ \mathcal T[{n,\mathbf m}] \Bigg)(H_{n\mathbf m}) \Bigg].
 \end{align*}
 This rests upon the simple fact that $\mathbf j+\mathbf d_s\in\mathbf C(\mathbf j)$ if and only if $\chi_{s+1}(\mathbf j)=1$.
 Then by using the definition of $ \mathcal K[{i, s,\mathbf j}]$, we obtain
 \begin{align*}
   \Phi_{n\mathbf j}& = \mathcal T[{n,\mathbf j}](H_{n\mathbf j})-r\sum_{s=1}^{n-1} \chi_{s+1}(\mathbf j) \xi_{s}(j_s+1) \mathcal T[{n,\mathbf j}] \circ \Bigg[ \mathcal T[{n,\mathbf j+\mathbf d_s}](H_{n(\mathbf j+\mathbf d_s)})\\
  &+ \sum_{\mathbf m\in \mathbf C(\mathbf j+\mathbf d_s)} \sum_{\mathbf p\in \mathbf P(\mathbf m,\mathbf j+\mathbf d_s)} \Bigg(\prod_{u=1}^{\ell(\mathbf m,\mathbf j+\mathbf d_s)} \mathcal K[{n, p_u,\mathbf m-\sum_{v=1}^u \mathbf d_{p_v} }] \circ \mathcal T[{n,\mathbf m}] \Bigg)(H_{n\mathbf m}) \Bigg]\\
  &=\mathcal T[{n,\mathbf j}](H_{n\mathbf j})-r\sum_{s=1}^{n-1} \chi_{s+1}(\mathbf j) \xi_{s}(j_s+1) \mathcal T[{n,\mathbf j}](\Phi_{n(\mathbf j+\mathbf d_s)}),
 \end{align*}
 where we have used the induction hypothesis in the last equality.  The result therefore follows. 
\end{proof}}

\textbf{In the following, we will fix 
\begin{align}\eqlab{reqn}
r=\frac{1}{2K(n-1)}.
\end{align}
}

{We then have:
\begin{lemma}
Consider 
\begin{align*}
 \Phi_n(w,z) = \sum_{\mathbf j\,:\,(n,\mathbf j)\notin \mathcal R} \Phi_{n\mathbf j}(w)z^{\mathbf j},
\end{align*}
with $\Phi_{n\mathbf j}$ given by \eqref{Phinsol} for all $(n,\mathbf j)\notin \mathcal R$ and suppose that $H_n=\sum_{\mathbf j\in \mathbb N_0^n} H_{n\mathbf j}(w)z^{\mathbf j}\in \mathcal G_k\{z\}$. Then 
 \begin{align}
  \Vert \Phi_n \Vert_{\mu,k}\le 2K\Vert H_{n}\Vert_{\mu,k}\quad \mbox{and}\quad \Vert \partial_{z_q} \Phi_n \Vert_{\mu,k}\le 2K \mu  \Vert H_n\Vert_{\mu,k} \quad \forall\,q\in\{1,\ldots,n\}.\eqlab{this0}
  \end{align}
\end{lemma}
\begin{proof}
 First, from \eqref{Phinsol} we directly obtain that
 \begin{align*}
  \Vert\Phi_{n\mathbf j}\Vert&\le \frac{K}{1+\vert \mathbf j\vert} \left(\Vert H_{n\mathbf j}\Vert+\sum_{\mathbf m\in \mathbf C(\mathbf j)} \sum_{\mathbf p\in \mathbf P(\mathbf m,\mathbf j)}  {(rK)^{\ell(\mathbf m,\mathbf j)}} \Vert H_{n\mathbf m}\Vert\right)\\
  &\le K\left(\Vert H_{n\mathbf j}\Vert+\sum_{\mathbf m\in \mathbf C(\mathbf j)} \sum_{\mathbf p\in \mathbf P(\mathbf m,\mathbf j)}  {\left(\frac{1}{2(n-1)}\right)^{\ell(\mathbf m,\mathbf j)}} \Vert H_{n\mathbf m}\Vert\right)\quad \forall\,\mathbf j\,:\,(n,\mathbf j)\notin \mathcal R,
 \end{align*}
using \eqref{vmvj}, \eqref{estT}, \eqref{reqn} and
\begin{align*}
 \prod_{r=1}^{\ell(\mathbf m,\mathbf j)} \frac{rK((\mathbf m-\sum_{l=1}^r \mathbf d_{p_l})_{p_r}+1)  }{1+\vert\mathbf m\vert}=\prod_{r=1}^{\ell(\mathbf m,\mathbf j)} \frac{rK(\mathbf m-\sum_{l=1}^{r-1} \mathbf d_{p_l})_{p_r}  }{1+\vert\mathbf m\vert}\le (rK)^{\ell(\mathbf m,\mathbf j)}; 
\end{align*}
this last estimate follows from \eqref{vmvj} and \eqref{vertK}, see also \cite[p. 376]{de2020a}. The estimate for $\Vert \Phi_n\Vert$ then follows directly from the computations on \cite[p. 377]{de2020a}.
When comparing with \cite{de2020a} notice that in our case we have
\begin{align*}
 r_m: = \mu^{-1} r^{m}\quad \forall\,m\in \{1,\ldots,n\}\quad \Longrightarrow\quad \sum_{m=1}^{n-1}\frac{r_{m+1}}{r_m} = (n-1) r = \frac{1}{2K},
\end{align*}
recall \eqref{reqn}. ($r_m$ are quantities from \cite{de2020a} and Propositon 3.20 in this reference assumes that $\sum_{m=1}^{n-1}\frac{r_{m+1}}{r_m}\le \frac{1}{2K}$; the constant relating to $K$ is denoted by $\mathcal N^{-1}$ in \cite{de2020a}). 
 
 For the estimate for $\Vert\frac{\partial}{\partial z_q} \Phi_n\Vert$, we proceed completely analogously using \eqref{Philzq}, repeated here for convenience 
 \begin{align*}
  \frac{\partial}{\partial z_q} \Phi_n = \sum_{\mathbf j\,:\,(n,\mathbf j)\notin \mathcal R} \Phi_{n\mathbf j}j_q z^{\mathbf j-\mathbf e_q},
 \end{align*}
 and therefore by \eqref{Phinsol}:
\begin{align*}
 \Vert\Phi_{n\mathbf j} j_q\Vert&\le \frac{Kj_q}{1+\vert \mathbf  j\vert} \left(\Vert H_{n\mathbf j}\Vert+\sum_{\mathbf m\in \mathbf C(\mathbf j)} \sum_{\mathbf p\in \mathbf P(\mathbf m,\mathbf j)}  {(rK)^{\ell(\mathbf m,\mathbf j)}} \Vert H_{n\mathbf m}\Vert\right)\\
  &\le K\left(\Vert H_{n\mathbf j}\Vert+\sum_{\mathbf m\in \mathbf C(\mathbf j)} \sum_{\mathbf p\in \mathbf P(\mathbf m,\mathbf j)}  {(rK)^{\ell(\mathbf m,\mathbf j)}} \Vert H_{n\mathbf m}\Vert\right)\quad \forall\,\mathbf j\,:\,(n,\mathbf j)\notin \mathcal R.
\end{align*}
\end{proof}}
{We are now in a position to complete the proof of \lemmaref{LHSinv} in the case where $\Xi\ne 0$:
\begin{proof}[Completing the proof of \lemmaref{LHSinv}]
We will show the following:
 \begin{align}
  \Vert \Phi\Vert_{\mu,k}\le 2^n K \Vert H\Vert_{\mu,k},\quad \Vert \frac{\partial}{\partial z_q} \Phi\Vert_{\mu,k}\le 2^n K \mu\Vert H\Vert_{\mu,k}\quad \forall\,q\in\{1,\ldots,n\}.\eqlab{final}
 \end{align}
 By \lemmaref{GijPhiij}, we only have to treat the $\Phi$-component of $\operatorname{LHS}^{-1}$ (as in the semi-simple case) and the estimates \eqref{final} will therefore complete the proof of \lemmaref{LHSinv}. 
By \eqref{this0} we have already established that \begin{align}
  \Vert \Phi_n \Vert_{\mu,k}\le 2K\Vert H_{n}\Vert_{\mu,k}\quad \mbox{and}\quad \Vert \partial_{z_q} \Phi_n \Vert_{\mu,k}\le 2K \mu  \Vert H_n\Vert_{\mu,k} \quad \forall\,q\in\{1,\ldots,n\}. \eqlab{this}
 \end{align}
Then in the same way (as described above \lemmaref{Phinsol}), we obtain from \eqref{eqni2} that
\begin{align}
 \Vert \Phi_i\Vert_{\mu,k} \le 2K \left(\Vert H_i\Vert_{\mu,k}+r\Vert \Phi_{i+1}\Vert_{\mu,k}\right)\quad \forall\,i\in\{1,\ldots,n-1\}.\eqlab{this2}
\end{align}
and 
\begin{align}
 \Vert \partial_{z_q} \Phi_i \Vert_{\mu,k}\le 2K \mu \left(\Vert H_i\Vert_{\mu,k}+r\Vert \Phi_{i+1}\Vert_{\mu,k}\right) \quad \forall\,q\in\{1,\ldots,n\}.\eqlab{this22}
\end{align}
We now claim that 
\begin{align*}
 \Vert \Phi_{n-i}\Vert_{\mu,k} \le 2^{i+1} K \Vert H\Vert_{\mu,k},\quad \Vert H\Vert_{\mu,k} = \sup_{i\in \{1,\ldots,n\}}\Vert H_i\Vert_{\mu,k}\quad \forall\,i\in \{0,\ldots,n-1\}.
\end{align*}
The proof is by induction. It is true for $i=0$ (base case) by \eqref{this}. Now for the induction step, we use \eqref{this2}$_{i\rightarrow n-i}$:
\begin{align*}
 \Vert \Phi_{n-i}\Vert_{\mu,k} &\le 2K  \left(\Vert H\Vert_{\mu,k}+r\Vert \Phi_{n-(i-1)}\Vert_{\mu,k}\right)\\
 &\le 2K \Vert H\Vert_{\mu,k} \left(1+ r 2^{i} K \right)\\
 &\le 2K \Vert H\Vert_{\mu,k} (1+2^{i-1})\\
 &\le 2K  \Vert H\Vert_{\mu,k} 2^i,
\end{align*}
using the induction hypothesis and $r\le \frac{1}{2K}$, recall \eqref{reqn}. Therefore
\begin{align*}
 \Vert \Phi_{n-i}\Vert_{\mu,k} \le 2^{i+1} K \Vert H\Vert_{\mu,k} = 2^{n} K \Vert H \Vert_{\mu,k}\quad \forall\,i\in\{0,\ldots,n-1\}.
\end{align*}
which completes the proof of the first estimate in \eqref{final}. The estimates for $\Vert \partial_{z_q} \Phi \Vert_{\mu,k}$ are identical using \eqref{this22}.
\end{proof}}

\section{Application to the zero-Hopf}\seclab{zeroHopf}
In this section, we consider a real analytic zero-Hopf singularity in $\mathbb R^3$ with eigenvalues $0,\pm ib$, $b>0$, of the linearization. Under the assumption that the origin is an isolated singularity, we obtain the following pre-normal form  (see \eqref{prenormalform}):
\begin{align}
 \frac{1}{k} x^{k+1} \frac{dy}{dx} = A y +x f(x,y), \quad A = \operatorname{diag}(-ib,ib),\eqlab{prenormalform2}
\end{align}
for some $k\in \mathbb N$ with $y=(y_1,y_2)\in B^2(R)$, and where $f=(f_1,f_2)$ satisfies the following property
\begin{align}\eqlab{propP}
 (\mathrm P):\quad f_2(x,y_1,\overline y_1)=\overline{f_1(x,y_1,\overline y_1)}\quad \forall\,x\in \mathbb R\cap B(\nu),\,y_1\in B(R).
\end{align}
Here $\overline z$ denotes the complex conjugate of $z\in \mathbb C$.
Indeed, to realize the normal form \eqref{prenormalform2} satisfying the property $(\mathrm P)$, see \eqref{propP}, we first proceed as in \cite{bonckaert2008a} to write the system in the pre-normal form
\begin{align}
 \frac{1}{k} x^{k+1} \frac{du}{dx} = \textnormal{B} u +x h(x,u),\eqlab{prenormalformB}
\end{align}
for some $k\in \mathbb N$. Here $u=(u_1,u_2)\in B^2(R)$, $h=(h_1,h_2)$, $h=(h_1,h_2)$ is a pair of real analytic functions defined locally near the origin, and
\begin{align*}
 \textnormal{B} = \begin{pmatrix}
      0 & b\\
      -b & 0
     \end{pmatrix},\quad b>0.
\end{align*}
We then obtain \eqref{prenormalform2} with \eqref{propP} by setting $y_1 := u_1+iu_2,\,y_2:=u_1-i u_2$ and 
\begin{align*}
f_1(x,y) := (h_1+ih_2)(x,u),\quad f_2(x,y) := (h_1-ih_2)(x,u).
\end{align*}

A consequence of property $(\mathrm P)$ is obviously that 
\begin{align*}
 y_2 = \overline y_1\quad \mbox{and}\quad x\in \mathbb R \cap B(\nu),
\end{align*}
defines an invariant set of \eqref{prenormalform2}.

\begin{proposition}\proplab{hopffinal}
 Consider the system \eqref{prenormalform2} with $f$ satisfying property $(\mathrm P)$. Then \thmref{main1} applies with both $\theta=0$ and $\theta=\pi$ and where $C>0$ and $\mathcal R\subset \{1,2\}\times \mathbb N_0^2$ are so that
 \begin{align}
\rsp{\begin{cases} 
  \vert j_1 -(j_2+1)\vert \le C (1+\vert \mathbf j\vert) \quad \Longleftrightarrow \quad {(1,\mathbf j)\in \mathcal R},\\
    \vert j_2 -(j_1+1)\vert \le C(1+\vert \mathbf j\vert) \quad \Longleftrightarrow \quad { (2,\mathbf j)\in \mathcal R},
 \end{cases}}\eqlab{R1R2}
\end{align}
 so that in both cases ($\theta=0$ and $\theta=\pi$) the resulting $g$- and $\phi$-functions of \thmref{main1} (which we denote by $g^\theta$ and $\phi^\theta$, $\theta=0,\pi$, respectively) satisfy property $(\mathrm P)$: 
 \begin{align*}
 g_2^\theta(x,z_1,\overline z_1)= \overline{g_1^\theta(x,z_1,\overline z_1)}\quad \mbox{and} \quad \phi_2^\theta(x,z_1,\overline z_1)= \overline{ \phi_1^\theta(x,z_1,\overline z_1)} \quad \forall\, x\in \mathbb R\cap \omega_k,\,z_1\in B(R).
 \end{align*}
\end{proposition}
\begin{proof}
First, we consider \eqref{aux} and suppose that $H=(H_1,H_2)$ satisfies property $(\mathrm P)$:
\begin{align*}
 H_2(w,z_1,\overline z_1) = \overline{H_1(w,z_1,\overline z_1)} \quad \forall\,w\in \mathbb R\cap \Omega,\,z_1\in B(R),
\end{align*}
or equivalently at the level of the power series expansion of $H$ with respect to $z=(z_1,z_2)$:
\begin{align*}
 H_{2j_1j_2}(w) = \overline{H_{1j_2j_1}(w)}\quad \forall\,w\in \mathbb R\cap \Omega,\, \mathbf j=(j_1,j_2)\in \mathbb N_0^2.
\end{align*}
It then easily follows that the inverse $\operatorname{LHS}^{-1}$, defined by \eqref{LHSinv1} and \eqref{LHSinv2} with $A=\operatorname{diag}(-ib,ib)$, preserves the property $(\mathrm P)$. Specifically, from \eqref{LHSinv2} with $\lambda_1=-ib,\lambda_2= ib$, we have that
\begin{align*}
 \overline{\Phi_{1j_1j_2}(w)} &= -\frac{\overline{ H_{1j_1j_2}(w)}}{ib+(j_1-j_2 )ib-w^k}=-\frac{H_{2j_2j_1}(w)}{ib-(j_2-j_1 )ib-w^k}=\Phi_{2j_2j_1}\quad \forall\,\mathbf j=(j_1,j_2)\in \mathbb N_0^2,
\end{align*}
when $w\in \mathbb R$. In other words, let $\mathcal G_{k}^n\{z\}(\mathrm P)\subset \mathcal G_{k}^n\{z\}$ and  $\mathcal G_{k,1}^n\{z\}(\mathrm P)\subset \mathcal G_{k,1}^n\{z\}$ denote the subspaces defined by functions that satisfy the $(\mathrm P)$ property. These subspaces are clearly closed spaces and the restriction of $\operatorname{LHS}^{-1}$: 
\begin{align*}
 \operatorname{LHS}^{-1}\vert_{\mathrm P } \,:\,\mathcal G_{k}^n\{z\}(\mathrm P)\rightarrow \mathcal G_{k}^n\{z\}(\mathrm P)\times \mathcal G_{k,1}^n\{z\}(\mathrm P),
\end{align*}
is well-defined and uniformly bounded with respect to $\mu$. A similar result holds true for $\operatorname{RHS}$, see \eqref{rhs}, insofar that it maps $(G,\Phi)\in \mathcal G_{k}^n\{z\}(\mathrm P)\times \mathcal G_{k,1}^n\{z\}(\mathrm P)$ to $\operatorname{RHS}(G,\Phi)\in \mathcal G_{k}^n\{z\}(\mathrm P)$; for this we just use that for $w\in \mathbb R$ then
$( \overline {H\star_k G})(w) = (\overline {H} \star_k \overline{G})(w)$ for all $H,G\in \mathcal G_{k}$. In this way, 
\propref{final} holds true on a closed subset of $\mathcal G_{k}^n\{z\}(\mathrm P)\times \mathcal G_{k,1}^n\{z\}(\mathrm P)$. To complete the proof, we notice that 
\begin{align*}
 \overline{\mathcal L_{\theta,k}(H)(x)} = \int_0^{\theta \infty} \overline{H}(w)e^{-w^k/x^k} k dw=\mathcal L_{\theta,k}(\overline H)(x)\quad \forall\,H\in \mathcal G_{k},
\end{align*}
when $x\in \mathbb R\cap \omega_k$ and the integration is along the positive real axis ($\theta=0$) or the negative real axis ($\theta=\pi)$. 
\end{proof}

\begin{cor}\corlab{zeroHopf}
 Consider the assumptions of \propref{hopffinal} and fix any $N\in \mathbb N$. Then there is a choice of $C(N)>0$ and $\mathcal R=\mathcal R(N)$ such that
 \begin{align*}
  g_1(x,z) &= z_1 \left(\sum_{j=0}^{N-1} h_{10}(x)(z_1 z_2)^j+(z_1z_2)^{N} h_{11}(x,z_1,z_2)\right),\\
  g_2(x,z) &=z_2 \left(\sum_{j=0}^{N-1} h_{20}(x)(z_1 z_2)^j+(z_1z_2)^{N} h_{21}(x,z_1,z_2)\right),
 \end{align*}
 where $h_{20}(x)=\overline{h_{10}(x)}$ and $h_{21}(x,z_1,\overline z_1)=\overline{ h_{11}(x,z_1,\overline z_1)}$ for all $x\in \mathbb R\cap \omega_k$ and $z_1\in B(R)$. 
In particular, we have the following real normal form:
\begin{align*}
\frac{1}{k}x^{k+1} \frac{dr}{dx} &= x\left(\sum_{j=0}^{N-1} \operatorname{Re}(h_{10}(x))r^{2j}+r^{2N} \operatorname{Re}\left(h_{11}(x,z_1,\overline z_1) \right)\right) r,\\
\frac{1}{k}x^{k+1} \frac{d\phi}{dx} & = - b+x \left(\sum_{j=0}^{N-1} \operatorname{Im}(h_{10}(x))r^{2j}+r^{2N} \operatorname{Im}\left(h_{11}(x,z_1,\overline z_1)\right)\right),
\end{align*}
for $x\in \mathbb R\cap \omega_k$, in the polar coordinates $r,\phi$ defined by $z_1=re^{i\phi}$. 

\end{cor}
\begin{proof}
\rsp{We simply apply \propref{hopffinal} with $C>0$ small enough. In further details, let $\mathcal R_1$ and $\mathcal R_2$ denote the projections of $\mathcal R$ onto $i=1$ and $i=2$, respectively. We then notice by \eqref{R1R2} that $\mathcal R_1=\{(j_1,j_2)\in \mathbb N_0^2\,:\,j_1=j_2+1\}$ and $\mathcal R_2=\{(j_1,j_2)\in \mathbb N_0^2\,:\,j_2=j_1+1\}$ for  $C= 0$. Consequently, for any $N\in \mathbb N$, there is $C>0$ small enough such that
\begin{align*}
 \mathcal R_1\cap \{(j_1,j_2)\in \mathbb N_0^2\,:\,j_2\le N-1 \mbox{  or  } j_1\le N\} &= \{(j_1,j_2)\in \mathbb N_0^2\,:\,j_1=j_2+1 \mbox{  and  } j_2\le N-1\},\\ \mathcal R_2\cap \{(j_1,j_2)\in \mathbb N_0^2\,:\,j_1\le N-1 \mbox{  or  } j_2\le N \} &= \{(j_1,j_2)\in \mathbb N_0^2\,:\, j_2=j_1+1 \mbox{  and  } j_1\le N-1\}.
\end{align*}
%
%
For $\mathcal R_1$ this simply means that it is a subset of the purple region in \figref{R1}. For $\mathcal R_2$, we reflect the region around the bisector. This gives the expansion of $g_1$ and $g_2$ upon application of  \propref{hopffinal}.}
\end{proof}

In the normal form of \corref{zeroHopf}, $g(x,0)=0$ for all $x\in \omega_k$. In other words, $z=0$ defines an invariant set of \eqref{finalnormalform} for $x\in \omega_k$. Given that $\phi$ satisfies the property $(\mathrm P)$ when either $\theta=0$ or $\theta=\pi$, it follows that in the $u$-coordinates, recall \eqref{prenormalformB}, the resulting invariant manifolds
\begin{align}
 \theta=0:\quad \begin{cases} 
u_1& = \frac12 x (\phi_1^{0}(x,0)+\phi_2^{0}(x,0))\\
u_2& = \frac{1}{2i} x (\phi_1^{0}(x,0)-\phi_2^{0}(x,0))
 \end{cases}\quad \mbox{for}\quad x\in \omega_k(\nu,0,\alpha),\eqlab{inv1}
\end{align}
and
\begin{align}
 \theta=\pi:\quad\begin{cases} 
u_1& = \frac12 x (\phi_1^{\pi}(x,0)+\phi_2^{\pi}(x,0))\\
u_2& = \frac{1}{2i} x (\phi_1^{\pi}(x,0)-\phi_2^{\pi}(x,0))
 \end{cases}\quad \mbox{for}\quad x\in \omega_k(\nu,\pi,\alpha),\eqlab{inv2}
\end{align}
are \textit{real} when $x\in \mathbb R\cap \omega_k$.

%

In the case $k=1$, the invariant manifolds \eqref{inv1} and \eqref{inv2} correspond to the invariant manifolds of \cite[Theorem 2.13]{baldom2013a} used in the analysis of the analytic unfolding of the zero-Hopf (notice that the authors use $s=x^{-1}$ as the independent variable, with $x$ being denoted $\overline z$ in \cite{baldom2013a}, see \cite[Eq. (6)]{baldom2013a} with $\mu=0$). It is important for the analysis in \cite{baldom2013a} that in the $k=1$ case the domains $\omega_k(\nu,0,\alpha)$ and $\omega_k(\theta,\pi,\alpha)$ have non-empty intersection along the imaginary axis ($\phi^0(\cdot,0)\ne \phi^\pi(\cdot,0)$ in general there), recall \figref{omega}; this does not hold true in general in the case where $k\ge 2$. It would be interesting in future work to analyze what implications this has for higher co-dimensions versions of the phenomena in \cite{baldom2013a}.

{The reference \cite{bittmann2018a} obtains a stronger normal form (where the nonlinearity only depends upon resonant monomials) in the case $k=1$ for resonant saddle-nodes with eigenvalues $-\lambda,\lambda,0$ under a separate condition on terms of the form $xy$.}

 \begin{figure}[ht!]
\begin{center}
{\includegraphics[width=.65\textwidth]{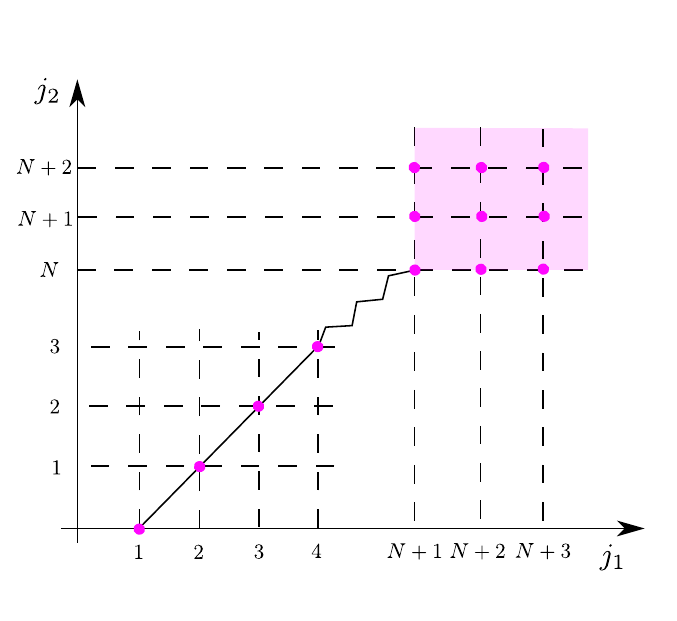}}
\end{center}
\caption{\rsp{For any $N\in \mathbb N$ there is a $C>0$ small enough so that the projection $\mathcal R_1$ of $\mathcal R$ onto $i=1$ is a subset of the purple region}. The case $i=2$ is obtained by reflecting the $i=1$ case around the bisector. }
\figlab{R1}
\end{figure}

\section{Discussion}\seclab{discussion}

In this final section, we would like to restrict our attention to the case where all nonzero eigenvalues belong to the Poincar\'e domain, i.e.~all of them belong to a common complex halfplane, for example the halfplane to the left of the imaginary axes.  Without a zero eigenvalue, reduction to normal form is well-known from a classical theorem of Poincar\'e: reduction to normal form can be done analytically and the normal form only contains a finite number of resonant monomials.

In \cite{bonckaert2008a}, their Theorem~3 was casted as  a possible generalization of this result, claiming that in the presence of one zero-eigenvalue, in the generic case $k=1$, one obtains $1$-summability of the normal form transformation in the real direction.  This is not true (it is not proven and even not true in general).  In fact, we would like to pose it as a research question:

\begin{problem} \thmlab{mainprob}
 Consider \eqref{prenormalform} satisfying \assumpref{Adiag} with $f$ being analytic near the origin and let all nonzero eigenvalues have strictly negative real part.  Then there exists a formal change of coordinates bringing the vector field in the form
 \begin{align}
 \frac{1}{k}x^{k+1} \frac{dz}{dx} = A z+ xg_{res}(x,z) + xg_{nonres}(x,z)\eqlab{finalnormal2}
\end{align}
where the Taylor series w.r.t.~$z$ of $g_{res}(x,z)$ only contains a finite number of resonant monomials of the form $xg_k(x)z^k$, and where $g_{nonres}(x,z)$ is ``as close to annihilation as possible''.  Try to capture the convergence or divergence properties of these transformations, possibly balancing the terms that can be annihilated in $g_{nonres}$.
\end{problem}

In the simplest setting the question relates to the analytic normalizing transformation 
\begin{align}
        x^2 \frac{dy}{dx} = -y + x f(x,y)\qquad \longrightarrow \qquad x^2 \frac{dz}{dx} = -z + x g(z).\eqlab{loray}
\end{align}
Because there is only one nonzero eigenvalue, it trivially belongs to the Poincar\'e domain.  There are no resonant terms here, and in the proposed normal form, the nonresonant terms are almost completely annihilated in the sense that $x.g_k$ only contains a single linear term in $x$ for all $k$.  Allowing nonresonant terms with linear coefficients allows Frank Loray to prove in \cite[Theorem 1]{loray2004a} that the normalizing transformation is convergent (and real analytic).  Later, the result has been generalized in \cite{schaefke2015a}.  In \cite{bonckaert2008a} the authors prove that the normal form with $xg(z)=g_1xz$ is obtainable in a Gevrey-1 way, but unlike what the authors claim typically not in a 1-summable way in the positive real direction.

We speculate that the general \textit{Banach space and Borel--Laplace approach} of the present paper may provide a route for such general results and we are presently working in this direction. Our intention is to ensure ({by adjusting $g_{res}$}) that all complex singularities of the normalizing transformation in the Borel plane are removable.  When all singularities disappear, the Laplace transform becomes convergent!  

 \newpage 
\bibliography{refs}
\bibliographystyle{plain}
\end{document}